\documentclass[11pt,reqno]{article}
\usepackage[english]{babel}
\usepackage{graphicx}
\usepackage[centertags]{amsmath}
\usepackage{amsfonts,amssymb,amsthm}
\usepackage{hyperref}
\usepackage[utf8]{inputenc} 
\usepackage{color}
\usepackage{anysize}
\marginsize{3.0cm}{3.0cm}{2.2cm}{2.2cm}

\let\OLDthebibliography\thebibliography
\renewcommand\thebibliography[1]{
  \OLDthebibliography{#1}
  \setlength{\parskip}{0pt}
  \setlength{\itemsep}{0pt plus 0.3ex} }

\numberwithin{equation}{section}

\theoremstyle{plain}
\newtheorem{theorem}{Theorem}[section]

\newtheorem{lemma}[theorem]{Lemma}

\theoremstyle{definition}
\newtheorem{remark}[theorem]{Remark}

\def\notina[#1]#2{\begingroup\def\thefootnote{\fnsymbol{footnote}}\footnote[#1]{#2}\endgroup}

\newcommand{\N}{{\mathbb N}}
\newcommand{\R}{{\mathbb R}}
\newcommand{\C}{{\mathbb C}}
\newcommand{\Z}{\mathbb Z}

\newcommand{\T}{{\mathbb T}}

\newcommand{\mA}{\mathcal{A}}
\newcommand{\mB}{\mathcal{B}}
\newcommand{\mC}{\mathcal{C}}
\newcommand{\mD}{\mathcal{D}}
\newcommand{\mE}{\mathcal{E}}
\newcommand{\mF}{\mathcal{F}}
\newcommand{\mG}{\mathcal{G}}

\newcommand{\mK}{\mathcal{K}}

\newcommand{\mM}{\mathcal{M}}
\newcommand{\mP}{\mathcal{P}}
\newcommand{\mR}{\mathcal{R}}

\renewcommand{\a}{\alpha}
\renewcommand{\b}{\beta}

\renewcommand{\d}{\delta}

\newcommand{\e}{\varepsilon}
\newcommand{\ph}{\varphi}
\newcommand{\lm}{\lambda}

\newcommand{\Om}{\Omega}

\newcommand{\gr}{\nabla}

\newcommand{\la}{\langle}
\newcommand{\ra}{\rangle}

\newcommand{\pa}{\partial}
\newcommand{\Lm}{\Lambda}

\newcommand{\be}{\begin{equation}}
\newcommand{\ee}{\end{equation}}
\newcommand{\bal}{\begin{align}} 
\newcommand{\eal}{\end{align}} 

\title{\Large{\textbf{On the normal form of the Kirchhoff equation}}}

\author{\small{Pietro Baldi, Emanuele Haus}}

\date{\small{\em In memory of Walter Craig}} 

\begin{document}

\maketitle


\begin{small} 
\noindent
\emph{Abstract}.
Consider the Kirchhoff equation
\begin{equation*}
\pa_{tt} u - \Delta u \Big( 1 + \int_{\T^d} |\gr u|^2 \Big) = 0
\end{equation*}
on the $d$-dimensional torus $\T^d$.
In a previous paper we proved that, after a first step of {\em quasilinear} normal form, the resonant cubic terms show an integrable behavior, namely they give no contribution to the energy estimates.
This leads to the question whether the same structure also emerges at the next steps of normal form. In this paper, we perform the second step and give a negative answer to the previous question: the quintic resonant terms give a nonzero contribution to the energy estimates. This is not only a formal calculation, as we prove that the normal form transformation is bounded between Sobolev spaces.

\medskip

\noindent
\emph{Keywords}.\ 
Kirchhoff equation, quasilinear wave equations, 
quasilinear normal forms.

\noindent
\emph{MSC2010}:  
35L72, 
35Q74, 
37J40, 
70K45. 
\end{small}


\section{Introduction}  
\label{sec:1}

We consider the Kirchhoff equation on the $d$-dimensional torus $\T^d$, $\T := \R / 2\pi \Z$ 
(periodic boundary conditions)  
\be \label{K1}
\pa_{tt} u - \Delta u \Big( 1 + \int_{\T^d} |\gr u|^2 \, dx \Big) = 0.
\ee
Equation \eqref{K1} is a quasilinear wave equation, 
and it has the structure of a Hamiltonian system 
\be \label{p1}
\begin{cases} 
\pa_t u = \gr_v H(u,v) = v, \\ 
\pa_t v = - \gr_u H(u,v) = \Delta u \Big( 1 + \int_{\T^d} |\gr u|^2 dx \Big),
\end{cases}
\ee
where the Hamiltonian is
\be \label{K2} 
H(u,v) = \frac12 \int_{\T^d} v^2 dx 
+ \frac12 \int_{\T^d} |\gr u|^2 dx 
+ \Big( \frac12 \int_{\T^d} |\gr u|^2 dx \Big)^2,
\ee
and $\gr_u H$, $\gr_v H$ are the gradients with respect to the 
real scalar product 
\begin{equation} \label{lara}
\la f,g \ra := \int_{\T^d} f(x) g(x) \, dx \quad \forall f,g \in L^2(\T^d, \R),
\end{equation}
namely $H'(u,v)[f,g] = \la \gr_u H(u,v) , f \ra + \la \gr_v H(u,v) , g \ra$
for all $u,v,f,g$.
More compactly, \eqref{p1} is 
\be \label{p2} 
\pa_t w = J \gr H(w),
\ee
where $w = (u,v)$, $\gr H = (\gr_u H, \gr_v H)$ and 
\be \label{def J} 
J = \begin{pmatrix} 0 & 1 \\ -1 & 0 \end{pmatrix}.
\ee
The Cauchy problem for the Kirchhoff equation is given by \eqref{K1} 
with initial data at time $t=0$
\begin{equation} \label{initial data 1} 
u(0,x) = \a(x), \quad 
u_t(0,x) = \b(x).
\end{equation}

Such a Cauchy problem is known to be locally well posed in time for initial data $(\alpha,\beta)$ in the Sobolev space $H^{\frac32}(\T^d)\times H^{\frac12}(\T^d)$ (see the work of Dickey \cite{Dickey 1969}). However, the conserved Hamiltonian \eqref{K2} only controls the $H^1\times L^2$ norm of the couple $(u,v)$. Since the local well-posedness has only been established in regularity higher than the energy space $H^1\times L^2$, it is not trivial to determine whether the solutions are global in time.
In fact, the question of global well-posedness for the Cauchy problem \eqref{K1}-\eqref{initial data 1} with periodic boundary conditions (or with Dirichlet boundary conditions on bounded domains of $\R^d$) has given rise to a long-standing open problem: while it has been known for eighty years, since the pioneering work of Bernstein \cite{Bernstein 1940}, that analytic initial data produce global-in-time solutions, it is still unknown whether the same is true for $C^\infty$ initial data, even of small amplitude.

For initial data of amplitude $\e$, the linear theory immediately gives existence of the solution over a time interval of the order of $\e^{-2}$. In \cite{Kold}, we performed one step of quasilinear normal form and established a longer existence time, of the order of $\e^{-4}$; indeed, all the cubic terms giving a nontrivial contribution to the energy estimates are erased by the normal form. One may wonder whether the same type of mechanism works also for (one or more) subsequent steps of normal form.

In this paper, we give a negative answer to such a question, as we explicitly compute the second step of normal form for the Kirchhoff equation on $\T^d$, erasing all the nonresonant terms of degree five. It turns out that, differently from what happens for cubic terms, the contribution to the energy estimates of the resonant terms of degree five is different from zero. This, of course, leaves open the question whether for small amplitude initial data the time of existence can be extended beyond the lifespan $\sim\e^{-4}$ 
(partial results in this direction are in preparation \cite{K boot}).
The presence of resonant terms of degree five that give a nontrivial contribution to the energy estimates can, however, be interpreted as a sign of non-integrability of the equation. Another interesting open question is whether these ``non-integrable'' terms in the normal form can somehow be used to construct ``weakly turbulent'' solutions pushing energy from low to high Fourier modes, in the spirit of the works \cite{CKSTT 2010}, \cite{Guardia Kaloshin 2015},  
\cite{Haus Procesi 2015}, \cite{Guardia Haus Procesi 2016}, \cite{GHHMP 2018} for the semilinear Schr\"odinger equation on $\T^2$. Proving existence of such solutions may be a very hard task, but one may at least hope to use the normal form that we compute in this paper to detect some genuinely nonlinear behavior of the flow, over long time-scales (as in \cite{Grebert Thomann 2012}, \cite{Haus Thomann 2013}) or even for all times (as in \cite{Haus Procesi CMP 2017}).

\subsection{Main result}

To give a precise statement of our main result, we introduce here the functional setting.

\smallskip
\noindent
\emph{Function space.}
On the torus $\T^d$, it is not restrictive to set the problem in the space of functions with zero average in space, for the following reason.
Given initial data $\a(x), \b(x)$, we split both them and the unknown $u(t,x)$ into the sum of a zero-mean function and the average term, 
\[
\a(x) = \a_0 + \tilde \a(x), \quad 
\b(x) = \b_0 + \tilde \b(x), \quad
u(t,x) = u_0(t) + \tilde u(t,x),
\]
where 
\[
\int_{\T^d} \tilde\a(x) \, dx = 0, \quad
\int_{\T^d} \tilde\b(x) \, dx = 0, \quad
\int_{\T^d} \tilde u(t,x) \, dx = 0 \quad \forall t.
\]
Then the Cauchy problem \eqref{K1}-\eqref{initial data 1} splits into two distinct, uncoupled Cauchy problems: one is the problem for the average $u_0(t)$, which is
\[
u_0''(t) = 0, \quad u_0(0) = \a_0, \quad u_0'(0) = \b_0
\]
and has the unique solution $u_0(t) = \a_0 + \b_0 t$; 
the other one is the problem for the zero-mean component $\tilde u(t,x)$, which is 
\[
\tilde u_{tt} - \Delta \tilde u \Big( \int_{\T^d} |\gr \tilde u|^2 \, dx \Big) = 0, 
\quad
\tilde u(0,x) = \tilde \a(x), \quad
\tilde u_t(0,x) = \tilde \b(x).
\]
Thus one has to study the Cauchy problem for the zero-mean unknown $\tilde u(t,x)$ 
with zero-mean initial data $\tilde \a(x), \tilde \b(x)$; 
this means to study \eqref{K1}-\eqref{initial data 1}
in the class of functions with zero average in $x$.

For any real $s \geq 0$, we consider the Sobolev space of zero-mean functions
\begin{align} \label{def:Hs}
H^s_0(\T^d,\C) & := \Big\{ u(x) = \sum_{j \in \Z^d \setminus \{ 0 \} } u_j e^{ij\cdot x} : 
u_j \in \C, \ 
\| u \|_s < \infty \Big\}, 
\\
\| u \|_s^2 & := \sum_{j \neq 0} |u_j|^2 |j|^{2s},
\end{align}
and its subspace 
\begin{equation} \label{def:Hs reale}
H^s_0(\T^d,\R) := \{ u \in H^s_0(\T^d,\C) : u(x) \in \R \}
\end{equation}
of real-valued functions $u$, 
for which the complex conjugates of the Fourier coefficients satisfy 
$\overline{u_j} = u_{-j}$.
For $s=0$, we write $L^2_0$ instead of $H^0_0$ the space of square-integrable functions with zero average.

\medskip

Let $m_1=1$ if the dimension $d=1$ and $m_1=2$ if $d\geq2$. For $s\geq0,\d>0$ denote
$$
B^s(\d) := \big\{ (u,v)\in H^{s + \frac12}_0(\T^d,\R) \times H^{s - \frac12}_0(\T^d,\R) \, : \, \max\{ \| u \|_{m_1 + \frac12} , \| v \|_{m_1 - \frac12} \} \leq \d \big\},
$$
$$
B^s_{\rm sym}(\d) := \{ (u,v)\in H^{s}_0(\T^d,\C) \times H^{s}_0(\T^d,\C) \, : \, u=\bar v, \, \| u \|_{m_1} \leq \d \}.
$$

\bigskip

In this paper we prove the following normal form result.

\begin{theorem}\label{thm:main}
There exists $\d > 0$ and a map $\Phi: B^{m_1}_{\rm sym}(\d) \to B^{m_1}(2\d)$, ``close to identity'' (see Remark \ref{close-to-id}), injective and conjugating system \eqref{p1} to
\[
\pa_t \begin{pmatrix} u \\ v \end{pmatrix} = W(u,v)
\]
(see \eqref{def W}-\eqref{W decomp} and the whole Section \ref{sec:second step} for the precise definition of $W$).
The transformation $\Phi$ maps $B^s_{\rm sym}(\d)$ to $B^s(2\d)$ for all $s\geq m_1$.
The transformed vector field $W$ is the sum of linear terms, cubic terms, quintic terms, and a remainder of homogeneity $\geq7$. The linear and cubic terms of $W$ give zero contribution to the energy estimates (i.e.\ the estimates for the time evolution of Sobolev norms), while the quintic terms give a nonzero contribution to the energy estimates (see \eqref{energy.est}-\eqref{18dic.4}).
\end{theorem}

\begin{remark}\label{close-to-id}
In Section \ref{sec:Lin} we will introduce the transformations $\Phi^{(1)}$ and $\Phi^{(2)}$, which symmetrize the system and introduce complex coordinates. These transformations are not close to identity. By saying that the map $\Phi$ is ``close to identity'' we mean that $\Phi=\Phi^{(1)}\circ\Phi^{(2)}\circ\Phi^{\rm next}$, where $\Phi^{\rm next}$ is bounded from $B^s_{\rm sym}(\d)$ to $B^s_{\rm sym}(2\d)$ for all $s\geq m_1$ and satisfies
$$
\| (\Phi^{\rm next} - {\rm Id}) (u,v) \|_s \leq C \| (u,v) \|_{m_1}^2 \| (u,v) \|_s.
$$
\end{remark}

\begin{remark}
There is a certain similarity between our computation and the one performed by Craig and Worfolk \cite{Craig Worfolk} for the normal form of gravity water waves. In both cases one deals with an equation whose vector field is strongly unbounded (quasilinear here, fully nonlinear in \cite{Craig Worfolk}) and in both cases the first steps of normal form show an ``integrable'' behavior, while after few steps some genuinely non-integrable terms show up.

However, there is an important difference: while the normal form computed in \cite{Craig Worfolk} is only the result of a formal computation, the transformation $\Phi$ that we construct here 
to put the Kirchhoff equation in normal form 
is a bounded transformation that is well defined between Sobolev spaces. 
This is obtained thanks to the ``quasilinear symmetrization'' performed in \cite{Kold}, 
following the strategy for quasilinear normal forms introduced by Delort in the papers \cite{Delort 2009}-\cite{Delort 2012} on quasilinear Klein-Gordon equation on $\T$.
\end{remark}

\subsection{Related literature}

Equation \eqref{K1} was introduced by Kirchhoff \cite{Kirchhoff 1876}  
to model the transversal oscillations of a clamped string or plate, 
taking into account nonlinear elastic effects.
The first results on the Cauchy problem \eqref{K1}-\eqref{initial data 1}
are due to Bernstein. In his 1940 pioneering paper \cite{Bernstein 1940}, 
he studied the Cauchy problem on an interval, with Dirichlet boundary conditions, 
and proved global wellposedness for analytic initial data $(\a,\b)$.

After that, the research on the Kirchhoff equation has been developed in various directions,
with a different kind of results on compact domains (bounded subsets of $\R^d$ with Dirichlet boundary conditions, or periodic boundary conditions $\T^d$) 
or non compact domains ($\R^d$ or ``exterior domains'' $\Om = \R^d \setminus K$, 
with $K \subset \R^d$ compact domain).

On $\R^d$, Greenberg and Hu \cite{Greenberg Hu 1980} in dimension $d=1$ 
and D'Ancona and Spagnolo \cite{D'Ancona Spagnolo 1993} in higher dimension 
proved global wellposedness with scattering for small initial data in weighted Sobolev spaces.  

On compact domains, 
dispersion, scattering and time-decay mechanisms are not available, 
and there are no results of global existence, nor of finite time blowup, 
for initial data $(\a,\b)$ of Sobolev, or $C^\infty$, or Gevrey regularity. 
The local wellposedness in the Sobolev class $H^{\frac32} \times H^{\frac12}$ 
has been proved by 
Dickey \cite{Dickey 1969} (see also Arosio and Panizzi \cite{Arosio Panizzi 1996}), 
Beyond the question about the global wellposedness for small data in Sobolev class, 
another open question concerns 
the local wellposedness in the energy space $H^1 \times L^2$ 
or in $H^s \times H^{s-1}$ for $1 < s < \frac32$. 

We also mention the recent results \cite{Baldi 2009},
\cite{Montalto 2017 KAM K forced},
\cite{Corsi Montalto 2018},
which prove the existence of time periodic or quasi-periodic solutions 
of time periodically or quasi-periodically forced Kirchhoff equations on $\T^d$,
using Nash-Moser and KAM techniques. 

For more details, generalizations
and other open questions, we refer to 
Lions \cite{Lions 1978}, to the surveys of 
Arosio \cite{Arosio 1993}, 
Spagnolo \cite{Spagnolo 1994}, 
Matsuyama and Ruzhansky \cite{Matsuyama Ruzhansky 2015}, and to other references in our previous paper \cite{Kold}.

\medskip

Concerning the normal form theory, and limiting ourselves to quasilinear PDEs on compact manifolds, we mention, in addition to the aforementioned papers of Delort \cite{Delort 2009}-\cite{Delort 2012}, the abstract result of Bambusi \cite{Bambusi 2005}
the recent literature on water waves by 
Craig and Sulem \cite{Craig Sulem 2016 BUMI},
Ifrim and Tataru \cite{Ifrim Tataru 2017},
Berti and Delort \cite{Berti Delort}, 
Berti, Feola and Pusateri \cite{Berti Feola Pusateri}-\cite{Berti Feola Pusateri nota},
and the work by Feola and Iandoli \cite{Feola Iandoli} on the quasilinear NLS on $\T$.

\bigskip

\noindent
\textbf{Acknowledgements}.
We thank Roberto Feola for some useful discussions on this subject.
This research is supported by the INdAM-GNAMPA Project 2019. 

\section{Linear transformations} \label{sec:Lin}
We start by recalling the first standard transformations in \cite{Kold},  
which transforms system \eqref{p1} into another one (see \eqref{syst uv}) 
where the linear part is diagonal, 
preserving both the real and the Hamiltonian structure of the problem.
These standard transformations are the symmetrization of the highest order 
and then the diagonalization of the linear terms.

\medskip

\emph{Symmetrization of the highest order.}
In the Sobolev spaces \eqref{def:Hs} of zero-mean functions, 
the Fourier multiplier 
\begin{equation*} 
\Lm := |D_x| : H^s_0 \to H^{s-1}_0, \quad  
e^{ij \cdot x} \mapsto |j| e^{ij \cdot x}
\end{equation*}
is invertible.   
System \eqref{p1} writes
\begin{equation} \label{1912.1}
\begin{cases} 
\pa_t u = v \\ 
\pa_t v = - ( 1 + \la \Lm u, \Lm u \ra ) \Lm^2 u,
\end{cases}
\end{equation}
where $\la \cdot , \cdot \ra$ is defined in \eqref{lara};
the Hamiltonian \eqref{K2} is
\[
H(u,v) = \frac12 \la v , v \ra + \frac12 \la \Lm u, \Lm u \ra 
+ \frac14 \la \Lm u, \Lm u \ra^2.
\]
To symmetrize the system at the highest order, 
we consider the linear, symplectic transformation 
\begin{equation} \label{1912.2}
(u,v) = \Phi^{(1)}(q,p) 
= ( \Lm^{-\frac12} q , \Lm^{\frac12} p). 
\end{equation}
System \eqref{1912.1} becomes 
\begin{equation} \label{1912.3}
\begin{cases}
\pa_t q = \Lm p \\ 
\pa_t p = - ( 1 + \la \Lm^{\frac12} q, \Lm^{\frac12} q \ra ) \Lm q,
\end{cases}
\end{equation}
which is the Hamiltonian system $\pa_t (q,p) = J \gr H^{(1)}(q,p)$ 
with Hamiltonian $H^{(1)} = H \circ \Phi^{(1)}$, namely
\begin{equation} \label{def H(1)}
H^{(1)}(q,p) 
= \frac12 \la \Lm^{\frac12} p, \Lm^{\frac12} p \ra 
+ \frac12 \la \Lm^{\frac12} q, \Lm^{\frac12} q \ra 
+ \frac14 \la \Lm^{\frac12} q, \Lm^{\frac12} q \ra^2, 
\quad 
J := \begin{pmatrix} 
0 & I \\ 
- I & 0 \end{pmatrix}.
\end{equation}
The original problem requires the ``physical'' variables $(u,v)$ to be real-valued; 
this corresponds to $(q,p)$ being real-valued too.
Also note that $\la \Lm^{\frac12} p, \Lm^{\frac12} p \ra = \la \Lm p, p\ra$.

\medskip

\emph{Diagonalization of the highest order: complex variables.}
To diagonalize the linear part $\pa_t q = \Lm p$, $\pa_t p = - \Lm q$
of system \eqref{1912.3}, we introduce complex variables. 

System \eqref{1912.3} and the Hamiltonian $H^{(1)}(q,p)$ in \eqref{def H(1)} 
are also meaningful, without any change, for \emph{complex} functions $q,p$. 
Thus we define the change of complex variables $(q,p) = \Phi^{(2)}(f,g)$ as
\begin{equation} \label{def Phi2}
(q,p) = \Phi^{(2)}(f,g) = \Big( \frac{f+g}{\sqrt2}, \frac{f-g}{i \sqrt2} \Big),
\qquad 
f = \frac{q + i p}{\sqrt2}, \quad  
g = \frac{q - i p}{\sqrt2}, 
\end{equation}
so that system \eqref{1912.3} becomes
\begin{equation} \label{syst uv}
\begin{cases}
\pa_t f = - i \Lm f - i \frac14 \la \Lm(f+g) , f+g \ra \Lm(f+g)
\\
\pa_t g = i \Lm g + i \frac14 \la \Lm(f+g) , f+g \ra \Lm(f+g)
\end{cases}
\end{equation}
where the pairing $\la \cdot , \cdot \ra$ denotes the integral of the product 
of any two complex functions
\begin{equation} \label{SPWC} 
\la w , h \ra := \int_{\T^d} w(x) h(x) \, dx 
= \sum_{j \in \Z^d \setminus \{ 0 \} } w_j h_{-j}, 
\quad w, h \in L^2(\T^d, \C).
\end{equation}
The map $\Phi^{(2)} : (f,g) \mapsto (q,p)$ in \eqref{def Phi2} 
is a $\C$-linear isomorphism of the space $L^2_0(\T^d,\C) \times L^2_0(\T^d,\C)$ 
of pairs of complex functions. 
When $(q,p)$ are real, $(f,g)$ are complex conjugate.
The restriction of $\Phi^{(2)}$ to the space
\begin{equation*} 
L^2_0(\T^d, c.c.) := 
\{ (f,g) \in L^2_0(\T^d,\C) \times L^2_0(\T^d,\C) : g = \overline{f} \}
\end{equation*}
of pairs of complex conjugate functions 
is an $\R$-linear isomorphism onto the space $L^2_0(\T^d,\R) \times L^2_0(\T^d,\R)$ 
of pairs of real functions. 
For $g = \overline{f}$, the second equation in \eqref{syst uv} is redundant, 
being the complex conjugate of the first equation. 
In other words, system \eqref{syst uv} has the following ``real structure'': 
it is of the form 
\begin{equation*} 
\pa_t \begin{pmatrix} f \\ g \end{pmatrix} 
= \mF(f,g) = \begin{pmatrix} \mF_1(f,g) \\ \mF_2(f,g) \end{pmatrix}
\end{equation*}
where the vector field $\mF(f,g)$ satisfies 
\begin{equation} \label{real vector field}
\mF_2(f, \overline{f}) = \overline{ \mF_1(f, \overline{f}) }.
\end{equation}
Under the transformation $\Phi^{(2)}$, the Hamiltonian system \eqref{1912.3} 
for complex variables $(q,p)$ becomes \eqref{syst uv}, which is the Hamiltonian system 
$\pa_t(f,g) = i J \gr H^{(2)}(f,g)$ with Hamiltonian 
$H^{(2)} = H^{(1)} \circ \Phi^{(2)}$, namely
\begin{equation*} 
H^{(2)}(f,g) = \la \Lm f, g \ra + \frac{1}{16} \la \Lm(f+g), f+g \ra^2,
\end{equation*} 
where $J$ is defined in \eqref{def H(1)}, 
$\la \cdot , \cdot \ra$ is defined in \eqref{SPWC},
and $\gr H^{(2)}$ is the gradient with respect to $\la \cdot , \cdot \ra$.
System \eqref{1912.3} for real $(q,p)$ (which corresponds to the original Kirchhoff equation)
becomes system \eqref{syst uv} restricted to the subspace $L^2_0(\T^d,c.c.)$ where 
$g = \overline{f}$. 

To complete the definition of the function spaces, 
for any real $s \geq 0$ we define 
\begin{equation*} 
H^s_0(\T^d,c.c.) := \{ (f,g) \in L^2_0(\T^d,c.c.) : f,g \in H^s_0(\T^d,\C) \}.
\end{equation*}

\section{Diagonalization of the order one}
In \cite{Kold} (Section 3) the following global transformation $\Phi^{(3)}$ is constructed.
Its effect is to remove the unbounded operator $\Lm$ from the ``off-diagonal'' terms of the equation, 
namely those terms coupling $f$ and $\bar f$.

\begin{lemma}[Lemma 3.1 of \cite{Kold}] 
\label{lemma:Phi3 inv}
Let $\Phi^{(3)}$ be the map 
\begin{equation} \label{def Phi3}
\Phi^{(3)}(\eta,\psi) = \mM(\eta,\psi) \begin{pmatrix} \eta \\ \psi \end{pmatrix}, 
\end{equation}
where $\mM(\eta,\psi)$ is the matrix
\begin{equation} \label{def M(eta,psi)}
\mM(\eta,\psi) := \frac{1}{\sqrt{1-\rho^2(P(\eta,\psi))}} 
\begin{pmatrix} 1 & \rho(P(\eta,\psi)) \\ 
\rho(P(\eta,\psi)) & 1 \end{pmatrix},
\end{equation}
$\rho$ is the function
\begin{equation} \label{def rho}
\rho(x) := \frac{- x}{1 + x + \sqrt{1+2x}}\,,
\end{equation}
$P$ is the functional
\begin{equation} \label{Q ph Q}
P(\eta,\psi) 
:= \ph(Q(\eta,\psi)), \qquad 
Q(\eta,\psi) := \frac{1}{4} \la \Lm (\eta + \psi) , \eta + \psi \ra,
\end{equation}
and $\ph$ is the inverse of the function $x \mapsto x \sqrt{1+2x}$, namely
\begin{equation} \label{def ph}
x \sqrt{1+2x} = y \quad \Leftrightarrow \quad x = \ph(y).
\end{equation}
Then, for all real $s \geq \frac12$, 
the nonlinear map $\Phi^{(3)} : H^s_0(\T^d, c.c.) \to H^s_0(\T^d, c.c.)$ 
is invertible, continuous, with continuous inverse 
\begin{equation*} 
(\Phi^{(3)})^{-1} (f,g) = \frac{1}{\sqrt{1 - \rho^2(Q(f,g))}} 
\begin{pmatrix} 1 & - \rho(Q(f,g)) \\ - \rho(Q(f,g)) & 1 \end{pmatrix} 
\begin{pmatrix} f \\ g \end{pmatrix}. 
\end{equation*}
For all $s \geq \frac12$, 
all $(\eta,\psi) \in H^s_0(\T^d,c.c.)$,  
one has 
\begin{equation*} 
\| \Phi^{(3)}(\eta,\psi) \|_s 
\leq C( \| \eta,\psi \|_{\frac12} ) \| \eta,\psi \|_s
\end{equation*}
for some increasing function $C$. 
The same estimate is satisfied by $(\Phi^{(3)})^{-1}$.
\end{lemma}

In \cite{Kold} it is proved that system \eqref{syst uv}, 
under the change of variable $(f,g) = \Phi^{(3)}(\eta, \psi)$, 
becomes 
\begin{equation} \label{syst 6 dic}
\begin{cases}
\pa_t \eta = - i \sqrt{1 + 2 P(\eta,\psi)} \, \Lm \eta
+ \dfrac{i}{4(1+2 P(\eta,\psi))} 
\Big( \la \Lm \psi, \Lm \psi \ra - \la \Lm \eta , \Lm \eta \ra \Big) \psi
\\
\pa_t \psi = i \sqrt{1 + 2 P(\eta,\psi)} \, \Lm \psi
+ \dfrac{i}{4(1+2 P(\eta,\psi))} 
\Big( \la \Lm \psi, \Lm \psi \ra - \la \Lm \eta , \Lm \eta \ra \Big) \eta.
\end{cases}
\end{equation}
Note that system \eqref{syst 6 dic} is diagonal at the order one, 
i.e.\ the coupling of $\eta$ and $\psi$ (except for the coefficients) 
is confined to terms of order zero. 
Also note that the coefficients of \eqref{syst 6 dic} are finite for $\eta,\psi \in H^1_0$, while the coefficients in \eqref{syst uv} are finite for $f,g \in H^{\frac12}_0$:  
the regularity threshold of the transformed system is $\frac12$ higher than before. 
The real structure is preserved, 
namely the second equation in \eqref{syst 6 dic} is the complex conjugate of the first one, 
or, in other words, the vector field in \eqref{syst 6 dic} satisfies property 
\eqref{real vector field}.

\medskip

\noindent
\emph{Quintic terms.} By Taylor's expansion, 
\begin{equation} \label{Taylor ph}
\ph(y) = y - y^2 + O(y^3) \quad (y \to 0).
\end{equation}
Hence 
\begin{align}
P(\eta,\psi) & = Q(\eta,\psi) - Q^2(\eta,\psi) + O(Q^3(\eta,\psi)), 
\notag \\
\frac{1}{1+2P(\eta,\psi)} 
& = 1 - 2 Q(\eta,\psi) + 6 Q^2(\eta,\psi) + O(Q^3(\eta,\psi)), 
\notag \\
\sqrt{1+2P(\eta,\psi)} 
& = 1 + Q(\eta,\psi) - \frac32 Q^2(\eta,\psi) + O(Q^3(\eta,\psi)).
\label{Taylor sqrt}
\end{align}

\medskip

\noindent
\emph{The transformed Hamiltonian.}
Even if $\Phi^{(3)}$ is not symplectic, nonetheless it could be useful to calculate 
the transformed Hamiltonian, because it is still a prime integral 
of the equation.
By definition \eqref{def rho}, one has 
\[
\frac{\rho(x)}{1 - \rho^2(x)} = \frac{-x}{2\sqrt{1+2x}}\,, \quad
\frac{1+\rho^2(x)}{1 - \rho^2(x)} = \frac{1+x}{\sqrt{1+2x}} \quad \forall x \geq 0.
\]
For $(f,g) = \Phi^{(3)}(\eta,\psi)$, one has
\[
\la \Lm f,g \ra 
= \frac{\rho(P(\eta,\psi))}{1-\rho^2(P(\eta,\psi))} 
\Big( \la \Lm \eta, \eta \ra + \la \Lm \psi, \psi \ra \Big)
+ \frac{1+\rho^2(P(\eta,\psi))}{1 - \rho^2(P(\eta,\psi))} 
\la \Lm \eta, \psi \ra
\]
and
\[
\frac{1}{16} \la \Lm(f+g), f+g \ra^2 
= Q^2(f,g)
= P^2(\eta,\psi). 
\]
Hence the new Hamiltonian $H^{(3)} := H^{(2)} \circ \Phi^{(3)}$ is
\begin{align*}
H^{(3)}(\eta,\psi) 
& = \frac{- P(\eta,\psi)}{2\sqrt{1+2 P(\eta,\psi)}} 
\Big( \la \Lm \eta, \eta \ra + \la \Lm \psi, \psi \ra \Big)
\\ & \quad \ 
+ \frac{1 + P(\eta,\psi)}{\sqrt{1+2 P(\eta,\psi)}} \la \Lm \eta, \psi \ra
+ P^2(\eta,\psi).
\end{align*}

\section{Normal form: first step} 
\label{sec:NF} 

The next step is the cancellation of the cubic terms contributing 
to the energy estimate. 
Following \cite{Kold}, we write \eqref{syst 6 dic} as
\begin{equation} \label{syst DBR}
\pa_t (\eta,\psi) = X(\eta,\psi) = \mD_1(\eta,\psi) + \mD_{\geq 3}(\eta,\psi) 
+ \mB_3(\eta,\psi) + \mR_{\geq 5}(\eta,\psi)
\end{equation}
where 
\begin{equation} \label{def mD geq 3}
\mD_1(\eta,\psi) := \begin{pmatrix} - i \Lm \eta \\ i \Lm \psi \end{pmatrix},
\quad 
\mD_{\geq 3}(\eta,\psi) := ( \sqrt{1 + 2 P(\eta,\psi)} \, - 1 ) \mD_1(\eta,\psi),
\end{equation} 
$\mB_3(\eta,\psi)$ is the cubic component of the bounded, off-diagonal term
\begin{equation} \label{def mB}
\mB_3(\eta,\psi) = 
\frac{i}{4} 
\Big( \la \Lm \psi, \Lm \psi \ra - \la \Lm \eta , \Lm \eta \ra \Big) 
\begin{pmatrix} \psi \\ \eta \end{pmatrix}
\end{equation}
and $\mR_{\geq 5}(\eta,\psi)$ is the bounded remainder of higher homogeneity degree
\begin{equation} \label{def mR geq 5}
\mR_{\geq 5}(\eta,\psi) = 
\frac{- i P(\eta,\psi)}{2 (1 + 2 P(\eta,\psi))} 
\Big( \la \Lm \psi, \Lm \psi \ra - \la \Lm \eta , \Lm \eta \ra \Big) 
\begin{pmatrix} \psi \\ \eta \end{pmatrix}.
\end{equation}
In \cite{Kold} the term $\mB_3$ 
(and not $\mD_{\geq 3}$, as it gives no contribution to the energy estimate)
is removed by the following normal form transformation.
Let 
\begin{equation} \label{def Phi4}
\Phi^{(4)}(w,z)
:= (I + M(w,z) ) \begin{pmatrix} w \\ z \end{pmatrix}, 
\end{equation}
\begin{equation} \label{def M} 
M(w,z) := \begin{pmatrix} 0 & A_{12}[w,w] + C_{12}[z,z] \\ 
A_{12}[z,z] + C_{12}[w,w] & 0 \end{pmatrix},
\end{equation}
where $A_{12}$, $C_{12}$ are the bilinear maps 
\begin{align}
\label{fix A12}
A_{12}[u, v] h
& 
:= \sum_{j,k \neq 0, \, |j| \neq |k|} 
u_j v_{-j} \frac{|j|^2}{8(|j| - |k|)} h_k e^{ik \cdot x},    
\\
C_{12}[u, v] h 
& 
:= \sum_{j,k \neq 0} u_j v_{-j} \frac{|j|^2}{8(|j| + |k|)} h_k e^{ik \cdot x}.
\label{fix C12}
\end{align}
For $d \in \N$, let
\begin{equation} \label{def m0}
m_0 = 1 \quad \text{if} \ d = 1, 
\qquad 
m_0 = \frac32 \quad \text{if} \ d \geq 2.
\end{equation} 

\begin{lemma}[Lemma 4.1 of \cite{Kold}]
\label{lemma:stime AC}
Let $A_{12}, C_{12}, m_0$ be defined in \eqref{fix A12}, \eqref{fix C12}, \eqref{def m0}.
For all complex functions $u,v,h$, all real $s \geq 0$, 
\begin{equation} \label{stima A12 basic}
\| A_{12}[u, v] h \|_s 
\leq \frac38 \| u \|_{m_0} \| v \|_{m_0} \| h \|_s,
\quad 
\| C_{12}[u, v] h \|_s 
\leq \frac{1}{16} \| u \|_1 \| v \|_1 \| h \|_s.
\end{equation}
\end{lemma}

The differential of $\Phi^{(4)}$ at the point $(w,z)$ is 
\begin{equation} \label{der Phi4}
(\Phi^{(4)})'(w,z) 
= (I + K(w,z)),  
\qquad 
K(w,z) = M(w,z) + E(w,z),
\end{equation}
where $M(w,z)$ is defined in \eqref{def M}, and 
\begin{equation} \label{def E}
E(w,z) \begin{pmatrix} \a \\ \b \end{pmatrix}
:= \begin{pmatrix} 2 A_{12}[w,\a] z + 2 C_{12}[z,\b] z \\
2 C_{12}[w,\a] w + 2 A_{12}[z,\b] w \end{pmatrix}.
\end{equation}
To estimate matrix operators and vectors in $H^s_0(\T^d,c.c.)$, 
we define $\| (w,z) \|_s := \| w \|_s = \| z \|_s$ 
for every pair $(w,z) = (w, \overline{w})$ of complex conjugate functions.  

\begin{lemma}[Lemma 4.2 of \cite{Kold}]
\label{lemma:stime MK} 
For all $s \geq 0$, all $(w,z) \in H^{m_0}_0(\T^d, c.c.)$, 
$(\a,\b) \in H^s_0(\T^d,c.c.)$ one has
\begin{align} \label{stima M}
\Big\| M(w,z) \begin{pmatrix} \a \\ \b \end{pmatrix} \Big\|_s 
& \leq \frac{7}{16} \| w \|_{m_0}^2 \|\a \|_s ,
\\
\Big\| K(w,z) \begin{pmatrix} \a \\ \b \end{pmatrix} \Big\|_s 
& \leq \frac{7}{16} \| w \|_{m_0}^2 \|\a \|_s 
+ \frac78 \| w \|_{m_0} \| w \|_s \| \a \|_{m_0} ,
\label{stima K}
\end{align}
where $m_0$ is defined in \eqref{def m0}. 
For $\| w \|_{m_0} < \frac12$, 
the operator $(I + K(w,z)) : H^{m_0}_0(\T^d, c.c.)$ $\to H^{m_0}_0(\T^d, c.c.)$ 
is invertible, with inverse 
\[
(I + K(w,z))^{-1} = I - K(w,z) + \tilde K(w,z), \quad 
\tilde K(w,z) := \sum_{n=2}^\infty (- K(w,z))^n,
\]
satisfying
\begin{equation*} 
\Big\| (I + K(w,z))^{-1} \begin{pmatrix} \a \\ \b \end{pmatrix} \Big\|_s 
\leq C (\| \a \|_s + \| w \|_{m_0} \| w \|_s \| \a \|_{m_0}),
\end{equation*}
for all $s \geq 0$, where $C$ is a universal constant.
\end{lemma}

The nonlinear, continuous map $\Phi^{(4)}$ is invertible in a ball around the origin. 

\begin{lemma}[Lemma 4.3 of \cite{Kold}] 
\label{lemma:Phi4 inv}
For all $(\eta, \psi) \in H^{m_0}_0(\T^d, c.c.)$ 
in the ball $\| \eta \|_{m_0} \leq \frac14$, 
there exists a unique $(w,z) \in H^{m_0}_0(\T^d, c.c.)$ such that 
$\Phi^{(4)}(w,z) = (\eta,\psi)$, with $\| w \|_{m_0} \leq 2 \| \eta \|_{m_0}$. 
If, in addition, $\eta \in H^s_0$ for some $s > m_0$, then $w$ also belongs to $H^s_0$, 
and $\| w \|_s \leq 2 \| \eta \|_s$.
This defines the continuous inverse map $(\Phi^{(4)})^{-1} : H^s_0(\T^d, c.c.) \cap 
\{ \| \eta \|_{m_0} \leq \frac14 \}$ $\to H^s_0(\T^d, c.c.)$.
\end{lemma}

\begin{lemma}[Lemma 4.4 of \cite{Kold}] 
\label{lemma:3001.1}
For all complex functions $u,v,y,h$, one has 
\begin{alignat}{2} 
\label{A12 C12 self-adj}
\la A_{12}[u,v] y , h \ra 
& = \la y , A_{12}[u,v] h \ra, 
\quad \ \ & 
\la C_{12}[u,v] y , h \ra 
& = \la y , C_{12}[u,v] h \ra,
\\
\label{A12 C12 conj}
\overline{A_{12}[u,v] y} 
& = A_{12}[ \overline{u}, \overline{v} ] \overline{y}, 
\quad & 
\overline{C_{12}[u,v] y} 
& = C_{12}[ \overline{u}, \overline{v} ] \overline{y},
\\
\label{A12 C12 commu}
[ A_{12} [u,v] , \Lm^s ]  
& = 0, 
\quad & 
[ C_{12} [u,v] , \Lm^s ] 
& = 0
\end{alignat}
where $\overline{u}$ is the complex conjugate of $u$, and so on.
Moreover, for all complex $w,z$, 
\begin{equation} \label{anticommu MD}
M(w,z) \mD_1 + \mD_1 M(w,z) = 0.
\end{equation}
\end{lemma}

Under the change of variables $(\eta,\psi) = \Phi^{(4)}(w,z)$,
it is proved in \cite{Kold} that system \eqref{syst 6 dic} becomes 
\begin{align} 
\pa_t \begin{pmatrix} w \\ z \end{pmatrix}
& = (I + K(w,z))^{-1} X(\Phi^{(4)}(w,z))
=: X^+(w,z) 
\notag \\
& = \big( 1 + \mP(w,z) \big) \mD_1(w,z) + X_3^+(w,z)
+ X_{\geq 5}^+(w,z)
\label{def X+}
\end{align}
where
\begin{equation} \label{def mP}
\mP(w,z) := \sqrt{1 + 2 P(\Phi^{(4)}(w,z))} \, - 1,
\end{equation}
$X_3^+(w,z)$ has components
\begin{align} \label{X3+ 1}
(X_3^+)_1(w,z) 
& := - \frac{i}{4} \sum_{j,k \neq 0,\, |k| = |j|} w_j w_{-j} |j|^2 z_k e^{ik \cdot x},
\\ 
(X_3^+)_2(w,z) 
& := \frac{i}{4} \sum_{j,k \neq 0,\, |k| = |j|} z_j z_{-j} |j|^2 w_k e^{ik \cdot x},
\label{X3+ 2}
\end{align}
and
\begin{align} 
X_{\geq 5}^+(w,z)
& := K(w,z) \big( I+K(w,z) \big)^{-1} \big( \mB_3(w,z) - X_3^+(w,z) \big)
+ \mR_{\geq 5}^+(w,z)
\notag \\ & \quad \ 
- \mP(w,z) \big( I+K(w,z) \big)^{-1} \big( \mB_3(w,z) - X_3^+(w,z) \big)
\label{def X+ geq 5}
\end{align}
with
\begin{align}
\label{def R+ geq 5}
\mR_{\geq 5}^+(w,z)
& := (I + K(w,z))^{-1} \mR_{\geq 5} ( \Phi^{(4)}(w,z))
+ [ \mB_3 ( \Phi^{(4)}(w,z)) - \mB_3(w,z) ]
\notag \\ & \qquad 
+ \big( - K(w,z) + \tilde K(w,z) \big) \mB_3 ( \Phi^{(4)}(w,z)),
\end{align}
$\mR_{\geq 5}$ defined in \eqref{def mR geq 5}.

\begin{lemma}[Lemma 4.5 of \cite{Kold}] 
\label{lemma:3001.2}
The maps $M(w, \overline{w})$, $K(w,\overline{w})$, 
and the transformation $\Phi^{(4)}$ 
preserve the structure of real vector field \eqref{real vector field}.
Hence $X^+$ defined in \eqref{def X+} satisfies \eqref{real vector field}.
\end{lemma}

The terms $(1+\mP) \mD_1$ and $X_3^+$ in \eqref{def X+}
give no contributions to the energy estimate, 
because, as one can check directly, 
\begin{equation*} 
\la \Lm^s (1 + \mP) (- i \Lm w) , \Lm^s z \ra 
+ \la \Lm^s w , \Lm^s (1 + \mP) i \Lm z \ra = 0
\end{equation*}
and
\begin{equation}
\la \Lm^s (X_3^+)_1, \Lm^s z \ra + \la \Lm^s w , \Lm^s (X_3^+)_2 \ra 
= 0.
\label{below is here}
\end{equation}
Similarly, also $\mP X_3^+$ gives no contribution to the energy estimate, 
because  
\[
\la \Lm^s (\mP X_3^+)_1, \Lm^s z \ra + \la \Lm^s w , \Lm^s (\mP X_3^+)_2 \ra 
= \mP \la \Lm^s (X_3^+)_1, \Lm^s z \ra 
+ \mP \la \Lm^s w , \Lm^s (X_3^+)_2 \ra = 0.
\]

\begin{lemma}[Lemma 4.6 of \cite{Kold}] 
\label{lemma:elementary}
For all $s \geq 0$, all pairs of complex conjugate functions $(w,z)$, 
one has
\begin{equation} \label{stima mB3 X3+}
\| \mB_3(w,z) \|_s \leq \frac12 \| w \|_1^2 \| w \|_s, 
\quad 
\| X_3^+(w,z) \|_s \leq \frac14 \| w \|_1^2 \| w \|_s, 
\end{equation}
and, for $\| w \|_{m_0} \leq \frac12$, for all complex functions $h$,  
\begin{align} \label{stima mP}
\| \mP(w,z) h \|_s & = \mP(w,z) \| h \|_s,
\quad 
0 \leq \mP(w,z) \leq C \| w \|_{\frac12}^2,
\\ 
\label{stima mR geq 5}
\| \mR_{\geq 5}(w,z) \|_s 
& \leq 2 P(w,z) \| \mB_3(w,z) \|_s 
\leq C \| w \|_{\frac12}^2 \| w \|_1^2 \| w \|_s
\end{align}
where $\mR_{\geq 5}$ is defined in \eqref{def mR geq 5}
and $C$ is a universal constant.
\end{lemma}

\begin{lemma}[Lemma 4.7 of \cite{Kold}]  
\label{lemma:4.7 of Kold} 
For all $s \geq 0$, all $(w,z) \in H^s_0(\T^d, c.c.) \cap 
H^{m_0}_0(\T^d, c.c.)$ with $\| w \|_{m_0} \leq \frac12$, one has 
\begin{equation} \label{stima X+ geq 5}
\| X_{\geq 5}^+ (w,z) \|_s 
\leq C \| w \|_1^2 \| w \|_{m_0}^2 \| w \|_s
\end{equation}
where $C$ is a universal constant.
\end{lemma}

\emph{Quintic terms.} 
Now we extract the terms of quintic homogeneity order 
from $X^+_{\geq 5}(w,z)$. 
Using \eqref{def X+ geq 5},
\eqref{def R+ geq 5},
\eqref{Taylor sqrt},
\eqref{Q ph Q},
\eqref{def Phi4},
we calculate 
\begin{equation} \label{X+ 57}
X^+_{\geq 5}(w,z) = \mP(w,z) X_3^+(w,z) + X^+_5(w,z) + X^+_{\geq 7}(w,z)
\end{equation}
where 
\begin{align} 
X^+_5(w,z) & := - K(w,z) X_3^+(w,z) 
- 3 Q(w,z) \mB_3(w,z) 
+ \mB_3'(w,z) M(w,z) \begin{pmatrix} w \\ z \end{pmatrix}
\label{X+ 5}
\end{align}
and $X^+_{\geq 7}(w,z)$ is defined in \eqref{X+ 57} by difference.
As already observed, the term $\mP(w,z) X_3^+(w,z)$ in \eqref{X+ 57} 
gives no contributions to the energy estimate.  
By \eqref{def X+}, \eqref{X+ 57}, the complete vector field is 
\begin{equation} \label{X+ new}
X^+(w,z) = (1 + \mP(w,z)) \big( \mD_1(w,z) + X_3^+(w,z) \big) 
+ X_5^+(w,z) + X_{\geq 7}^+(w,z).
\end{equation}
Moreover, adapting the proof of Lemma \ref{lemma:4.7 of Kold}, 
we obtain the following bounds. 

\begin{lemma} \label{lemma:20.05.2020} 
For all $s \geq 0$, all $(w,z) \in H^s_0(\T^d, c.c.) \cap 
H^{m_0}_0(\T^d, c.c.)$ with $\| w \|_{m_0} \leq \frac12$, one has 
\begin{equation*} 
\| X_{5}^+ (w,z) \|_s 
\leq C \| w \|_{m_0}^4 \| w \|_s,
\quad \ 
\| X^+_{\geq 7}(w,z) \|_s \leq C \| w \|_{m_0}^6 \| w \|_s,
\end{equation*}
where $C$ is a universal constant.
\end{lemma}

We analyze the terms in \eqref{X+ 5}.  
By \eqref{der Phi4}, \eqref{def E}, 
the first component of $K(w,z) X_3^+(w,z)$ is
\begin{align*}
( K(w,z) X_3^+(w,z) )_1 
& = A_{12}[w,w] (X_3^+)_2(w,z) 
+ C_{12}[z,z] (X_3^+)_2(w,z) 
\notag \\ & \quad \ 
+ 2 A_{12}[w, (X_3^+)_1(w,z)] z
+ 2 C_{12}[z, (X_3^+)_2(w,z)] z,
\end{align*}
and its second component is the conjugate of the first one.
Recalling \eqref{def mB}, 
the first component of the last term in \eqref{X+ 5} is
\begin{align*}
\Big( \mB_3'(w,z) M(w,z) \tbinom{w}{z} \Big)_1 
& = \frac{i}{2} \Big( \la \Lm z, \Lm \b \ra - \la \Lm w, \Lm \a \ra \Big) z 
+ \frac{i}{4} \Big( \la \Lm z, \Lm z \ra - \la \Lm w, \Lm w \ra \Big) \b
\end{align*}
with 
\[
\a = A_{12}[w,w] z + C_{12}[z,z]z, \quad 
\b = A_{12}[z,z] w + C_{12}[w,w]w,
\]
namely 
\begin{align*}
\Big( \mB_3'(w,z) M(w,z) \tbinom{w}{z} \Big)_1 
& = \frac{i}{2} \la \Lm z, A_{12}[z,z] \Lm w \ra z 
+ \frac{i}{2} \la \Lm z, C_{12}[w,w] \Lm w \ra z 
\\ & \quad \ 
- \frac{i}{2} \la \Lm w, A_{12}[w,w] \Lm z \ra z 
- \frac{i}{2} \la \Lm w, C_{12}[z,z] \Lm z \ra z 
\\ & \quad \ 
+ \frac{i}{4} \la \Lm z, \Lm z \ra A_{12}[z,z] w
+ \frac{i}{4} \la \Lm z, \Lm z \ra C_{12}[w,w] w
\\ & \quad \ 
- \frac{i}{4} \la \Lm w, \Lm w \ra A_{12}[z,z] w
- \frac{i}{4} \la \Lm w, \Lm w \ra C_{12}[w,w] w.
\end{align*}
In Fourier series, with all indices in $\Z^d \setminus \{ 0 \}$, one has 
\begin{align*}
A_{12}[w,w] (X_3^+)_2(w,z) 
& = \frac{i}{32} \sum_{\begin{subarray}{c} j,k,\ell \\ |j| \neq |k| = |\ell| \end{subarray}}
\frac{|j|^2 |\ell|^2}{|j| - |k|} w_j w_{-j} z_\ell z_{-\ell} w_k e^{ik\cdot x},
\\
C_{12}[z,z] (X_3^+)_2(w,z) 
& = \frac{i}{32} \sum_{\begin{subarray}{c} j,k,\ell \\ |k| = |\ell| \end{subarray}}
\frac{|j|^2 |\ell|^2}{|j| + |k|} z_j z_{-j} z_\ell z_{-\ell} w_k e^{ik\cdot x},
\\
A_{12}[w, (X_3^+)_1(w,z)] z
& = \frac{-i}{32} \sum_{\begin{subarray}{c} j,k,\ell \\ |\ell| = |j| \neq |k| \end{subarray}}
\frac{|j|^2 |\ell|^2}{|j| - |k|} w_j z_{-j} w_\ell w_{-\ell} z_k e^{ik\cdot x},
\\
C_{12}[z, (X_3^+)_2(w,z) ] z
& = \frac{i}{32} \sum_{\begin{subarray}{c} j,k,\ell \\ |j| = |\ell| \end{subarray}}
\frac{|j|^2 |\ell|^2}{|j| + |k|} z_j w_{-j} z_\ell z_{-\ell} z_k e^{ik\cdot x},
\end{align*}
\begin{align*}
Q(w,z) & = \frac14 \sum_j |j| (w_j w_{-j} + 2 w_j z_{-j} + z_j z_{-j}),
\\
(\mB_3(w,z))_1 & = \frac{i}{4} \sum_{j,k} |j|^2 (z_j z_{-j} - w_j w_{-j}) z_k e^{ik \cdot x},
\end{align*}
\[
\big( Q(w,z) \mB_3(w,z) \big)_1
= \frac{i}{16} \sum_{j,k,\ell} |\ell| |j|^2  
(w_\ell w_{-\ell} + 2 w_\ell z_{-\ell} + z_\ell z_{-\ell})
(z_j z_{-j} - w_j w_{-j}) z_k e^{ik \cdot x},
\]
\begin{align*}
\la \Lm z, A_{12}[z,z] \Lm w \ra z 
& = \frac{1}{8} \sum_{\begin{subarray}{c} j,k,\ell \\ |\ell| \neq |j| \end{subarray}}
\frac{|j|^2 |\ell|^2}{|\ell| - |j|} z_j w_{-j} z_\ell z_{-\ell} z_k e^{ik\cdot x},
\\
\la \Lm z, C_{12}[w,w] \Lm w \ra z 
& = \frac{1}{8} \sum_{\begin{subarray}{c} j,k,\ell \end{subarray}}
\frac{|j|^2 |\ell|^2}{|\ell| + |j|} z_j w_{-j} w_\ell w_{-\ell} z_k e^{ik\cdot x},
\\
\la \Lm w, A_{12}[w,w] \Lm z \ra z 
& = \frac{1}{8} \sum_{\begin{subarray}{c} j,k,\ell \\ |\ell| \neq |j| \end{subarray}}
\frac{|j|^2 |\ell|^2}{|\ell| - |j|} w_j z_{-j} w_\ell w_{-\ell} z_k e^{ik\cdot x},
\\
\la \Lm w, C_{12}[z,z] \Lm z \ra z 
& = \frac{1}{8} \sum_{\begin{subarray}{c} j,k,\ell \end{subarray}}
\frac{|j|^2 |\ell|^2}{|\ell| + |j|} w_j z_{-j} z_\ell z_{-\ell} z_k e^{ik\cdot x},
\\
\la \Lm z, \Lm z \ra A_{12}[z,z] w
& = \frac{1}{8} \sum_{\begin{subarray}{c} j,k,\ell \\ |k| \neq |j| \end{subarray}}
\frac{|j|^2 |\ell|^2}{|j| - |k|} z_j z_{-j} z_\ell z_{-\ell} w_k e^{ik\cdot x},
\\
\la \Lm z, \Lm z \ra C_{12}[w,w] w
& = \frac{1}{8} \sum_{\begin{subarray}{c} j,k,\ell \end{subarray}}
\frac{|j|^2 |\ell|^2}{|j| + |k|} w_j w_{-j} z_\ell z_{-\ell} w_k e^{ik\cdot x},
\end{align*}
\begin{align*}
\la \Lm w, \Lm w \ra A_{12}[z,z] w
& = \frac{1}{8} \sum_{\begin{subarray}{c} j,k,\ell \\ |k| \neq |j| \end{subarray}}
\frac{|j|^2 |\ell|^2}{|j| - |k|} z_j z_{-j} w_\ell w_{-\ell} w_k e^{ik\cdot x},
\\
\la \Lm w, \Lm w \ra C_{12}[w,w] w
& = \frac{1}{8} \sum_{\begin{subarray}{c} j,k,\ell \end{subarray}}
\frac{|j|^2 |\ell|^2}{|j| + |k|} w_j w_{-j} w_\ell w_{-\ell} w_k e^{ik\cdot x}.
\end{align*}
Thus the first component of the quintic term $X_5^+(w,z)$ is
\begin{align*}
(X^+_5(w,z))_1
& = 
- A_{12}[w,w] (X_3^+)_2(w,z) 
- C_{12}[z,z] (X_3^+)_2(w,z) 
\notag \\ & \quad \ 
- 2 A_{12}[w, (X_3^+)_1(w,z)] z
- 2 C_{12}[z, (X_3^+)_2(w,z)] z
\notag \\ & \quad \ 
- 3 \big( Q(w,z) \mB_3(w,z) \big)_1
\notag \\ & \quad \ 
+ \frac{i}{2} \la \Lm z, A_{12}[z,z] \Lm w \ra z 
+ \frac{i}{2} \la \Lm z, C_{12}[w,w] \Lm w \ra z 
\\ & \quad \ 
- \frac{i}{2} \la \Lm w, A_{12}[w,w] \Lm z \ra z 
- \frac{i}{2} \la \Lm w, C_{12}[z,z] \Lm z \ra z 
\\ & \quad \ 
+ \frac{i}{4} \la \Lm z, \Lm z \ra A_{12}[z,z] w
+ \frac{i}{4} \la \Lm z, \Lm z \ra C_{12}[w,w] w
\\ & \quad \ 
- \frac{i}{4} \la \Lm w, \Lm w \ra A_{12}[z,z] w
- \frac{i}{4} \la \Lm w, \Lm w \ra C_{12}[w,w] w
\end{align*}
and, in Fourier series, 
\begin{align*}
& (X^+_5(w,z))_1
= 
- \frac{i}{32} \sum_{\begin{subarray}{c} j,k,\ell \\ |j| \neq |k| = |\ell| \end{subarray}}
\frac{|j|^2 |\ell|^2}{|j| - |k|} w_j w_{-j} z_\ell z_{-\ell} w_k e^{ik\cdot x}
\notag \\ & \quad \ 
- \frac{i}{32} \sum_{\begin{subarray}{c} j,k,\ell \\ |k| = |\ell| \end{subarray}}
\frac{|j|^2 |\ell|^2}{|j| + |k|} z_j z_{-j} z_\ell z_{-\ell} w_k e^{ik\cdot x}
+ \frac{i}{16} \sum_{\begin{subarray}{c} j,k,\ell \\ |\ell| = |j| \neq |k| \end{subarray}}
\frac{|j|^2 |\ell|^2}{|j| - |k|} w_j z_{-j} w_\ell w_{-\ell} z_k e^{ik\cdot x}
\notag \\ & \quad \ 
- \frac{i}{16} \sum_{\begin{subarray}{c} j,k,\ell \\ |j| = |\ell| \end{subarray}}
\frac{|j|^2 |\ell|^2}{|j| + |k|} z_j w_{-j} z_\ell z_{-\ell} z_k e^{ik\cdot x}
\notag \\ & \quad \ 
- \frac{3 i}{16} \sum_{j,k,\ell} |\ell| |j|^2  
(w_\ell w_{-\ell} + 2 w_\ell z_{-\ell} + z_\ell z_{-\ell})
(z_j z_{-j} - w_j w_{-j}) z_k e^{ik \cdot x}
\notag \\ & \quad \
+ \frac{i}{16} \sum_{\begin{subarray}{c} j,k,\ell \\ |\ell| \neq |j| \end{subarray}}
\frac{|j|^2 |\ell|^2}{|\ell| - |j|} z_j w_{-j} z_\ell z_{-\ell} z_k e^{ik\cdot x}
+ \frac{i}{16} \sum_{\begin{subarray}{c} j,k,\ell \end{subarray}}
\frac{|j|^2 |\ell|^2}{|\ell| + |j|} z_j w_{-j} w_\ell w_{-\ell} z_k e^{ik\cdot x}
\notag \\ & \quad \ 
- \frac{i}{16} \sum_{\begin{subarray}{c} j,k,\ell \\ |\ell| \neq |j| \end{subarray}}
\frac{|j|^2 |\ell|^2}{|\ell| - |j|} w_j z_{-j} w_\ell w_{-\ell} z_k e^{ik\cdot x}
- \frac{i}{16} \sum_{\begin{subarray}{c} j,k,\ell \end{subarray}}
\frac{|j|^2 |\ell|^2}{|\ell| + |j|} w_j z_{-j} z_\ell z_{-\ell} z_k e^{ik\cdot x}
\notag \\ & \quad \ 
+ \frac{i}{32} \sum_{\begin{subarray}{c} j,k,\ell \\ |k| \neq |j| \end{subarray}}
\frac{|j|^2 |\ell|^2}{|j| - |k|} z_j z_{-j} z_\ell z_{-\ell} w_k e^{ik\cdot x}
+ \frac{i}{32} \sum_{\begin{subarray}{c} j,k,\ell \end{subarray}}
\frac{|j|^2 |\ell|^2}{|j| + |k|} w_j w_{-j} z_\ell z_{-\ell} w_k e^{ik\cdot x}
\notag \\ & \quad \ 
- \frac{i}{32} \sum_{\begin{subarray}{c} j,k,\ell \\ |k| \neq |j| \end{subarray}}
\frac{|j|^2 |\ell|^2}{|j| - |k|} z_j z_{-j} w_\ell w_{-\ell} w_k e^{ik\cdot x}
- \frac{i}{32} \sum_{\begin{subarray}{c} j,k,\ell \end{subarray}}
\frac{|j|^2 |\ell|^2}{|j| + |k|} w_j w_{-j} w_\ell w_{-\ell} w_k e^{ik\cdot x}.
\end{align*}

{\it Notation}. In the coefficients of the vector field $X^+_5$ there appear several denominators, 
which imply the corresponding restrictions on the indices $j,k,\ell$ to prevent the denominators from vanishing. From now on, we will stop indicating explicitly the restrictions on the indices in summations and adopt instead the convention $0/0=0$ in the coefficients. For instance, instead of
$$
\sum_{\begin{subarray}{c} j,k,\ell \\ |k| \neq |j| \end{subarray}}
\frac{|j|^2 |\ell|^2}{|j| - |k|} z_j z_{-j} w_\ell w_{-\ell} w_k e^{ik\cdot x}
$$
we will write
$$
\sum_{j,k,\ell}
\frac{|j|^2 |\ell|^2 (1-\d_{|j|}^{|k|})}{|j| - |k|} z_j z_{-j} w_\ell w_{-\ell} w_k e^{ik\cdot x}.
$$
In this example, when $|j|=|k|$ the denominator of the coefficient vanishes; the numerator also vanishes because of the factor $(1-\d_{|j|}^{|k|})$; this has to be interpreted as $\frac{|j|^2 |\ell|^2 (1-\d_{|j|}^{|k|})}{|j| - |k|}$ being zero when $|j|=|k|$.

\bigskip

We collect similar monomials, and we get that 
$(X^+_5(w,z))_1$ is the sum of the following eight terms:
\begin{align}
\label{Y(4)11}
Y^{(4)}_{11}[w,w,w,w] w 
& := - \frac{i}{32} \sum_{j,\ell,k} \frac{|j|^2 |\ell|^2}{|j| + |k|} 
w_{j} w_{-j} w_{\ell} w_{-\ell} w_k e^{ik \cdot x},
\\
Y^{(2)}_{11}[w,w,z,z] w
& := \frac{i}{32} \sum_{j,\ell,k} |j|^2 |\ell|^2 
\Big( \frac{- \d_{|\ell|}^{|k|} (1 - \d_{|j|}^{|k|})}{|j| - |k|}
+ \frac{1}{|j|+|k|} 
\notag \\ & \qquad 
- \frac{(1 - \d_{|\ell|}^{|k|})}{|\ell| - |k|} \Big) 
w_j w_{-j} z_\ell z_{-\ell} w_k e^{ik \cdot x},
\label{Y(2)11}
\\
\label{Y(0)11}
Y^{(0)}_{11}[z,z,z,z] w 
& := \frac{i}{32} \sum_{j,\ell,k} |j|^2 |\ell|^2 
\Big( \frac{- \d_{|\ell|}^{|k|} }{|j| + |k|}
+ \frac{(1 - \d_{|j|}^{|k|})}{|j| - |k|} \Big) 
z_j z_{-j} z_\ell z_{-\ell} w_k e^{ik \cdot x},
\end{align}
\begin{align}
\label{Y(4)12}
Y^{(4)}_{12} [w,w,w,w] z 
& := \frac{3 i}{16} \sum_{j,\ell,k} |j|^2 |\ell|  
w_j w_{-j} w_\ell w_{-\ell} z_k e^{ik \cdot x},
\\
Y^{(3)}_{12} [w,w,w,z] z
& := \frac{i}{16} \sum_{j,\ell,k} |j|^2 |\ell|
\Big( \frac{|\ell| \d_{|\ell|}^{|j|} (1 - \d_{|\ell|}^{|k|}) }{|\ell| - |k|}
+ 6 
\notag \\ & \qquad 
+ \frac{|\ell|}{|\ell| + |j|}
+ \frac{|\ell| (1 - \d_{|\ell|}^{|j|}) }{|\ell| - |j|} \Big) 
w_j w_{-j} w_\ell z_{-\ell} z_k e^{ik \cdot x},
\label{Y(3)12}
\\
\label{Y(2)12}
Y^{(2)}_{12} [w,w,z,z] z
& := \frac{3 i}{16} \sum_{j,\ell,k} |j| |\ell| 
(|j| - |\ell|) 
w_j w_{-j} z_\ell z_{-\ell} z_k e^{ik \cdot x},
\\
Y^{(1)}_{12} [w,z,z,z] z 
& := \frac{i}{16} \sum_{j,\ell,k} |j| |\ell|^2
\Big( \frac{- |j| \d_{|j|}^{|\ell|} }{|j| + |k|} 
- 6 + \frac{|j| (1 - \d_{|j|}^{|\ell|}) }{|\ell| - |j|}
\notag \\ & \qquad 
- \frac{|j|}{|\ell| + |j|} \Big)  
w_j z_{-j} z_\ell z_{-\ell} z_k e^{ik \cdot x},
\label{Y(1)12}
\\
\label{Y(0)12}
Y^{(0)}_{12} [z,z,z,z] z 
& := - \frac{3i}{16} \sum_{j,\ell,k} |j|^2 |\ell|
z_j z_{-j} z_\ell z_{-\ell} z_k e^{ik \cdot x}.
\end{align}
Symmetrizing in $j \leftrightarrow \ell$ when it is possible, 
we also have 
\begin{align}
\label{Y(4)11 symm}
Y^{(4)}_{11}[w,w,w,w] w 
& := - \frac{i}{64} \sum_{j,\ell,k} 
\Big( \frac{|j|^2 |\ell|^2}{|j| + |k|} + \frac{|j|^2 |\ell|^2}{|\ell| + |k|} \Big) 
w_{j} w_{-j} w_{\ell} w_{-\ell} w_k e^{ik \cdot x},
\\
Y^{(0)}_{11}[z,z,z,z] w 
& := \frac{i}{64} \sum_{j,\ell,k} |j|^2 |\ell|^2 
\Big( - \frac{\d_{|\ell|}^{|k|} + \d_{|j|}^{|k|}}{|j| + |\ell|}
\notag \\ & \quad \ 
+ \frac{(1 - \d_{|j|}^{|k|})}{|j| - |k|} 
+ \frac{(1 - \d_{|\ell|}^{|k|})}{|\ell| - |k|} \Big) 
z_j z_{-j} z_\ell z_{-\ell} w_k e^{ik \cdot x},
\label{Y(0)11 symm}
\\
\label{Y(4)12 symm}
Y^{(4)}_{12} [w,w,w,w] z 
& := \frac{3 i}{32} \sum_{j,\ell,k} |j| |\ell| (|j| + |\ell|) 
w_j w_{-j} w_\ell w_{-\ell} z_k e^{ik \cdot x},
\\
\label{Y(0)12 symm}
Y^{(0)}_{12} [z,z,z,z] z 
& := - \frac{3i}{32} \sum_{j,\ell,k} |j| |\ell| (|j| + |\ell|)
z_j z_{-j} z_\ell z_{-\ell} z_k e^{ik \cdot x}.
\end{align}

\section{Normal form: second step}
\label{sec:second step}

We consider a transformation of the form
\begin{equation} \label{def Phi5}
\begin{pmatrix} w \\ z \end{pmatrix}
= (I + \mM(u,v)) \begin{pmatrix} u \\ v \end{pmatrix} =: \Phi^{(5)}(u,v),
\end{equation}
where $\mM(u,v)$ is a matrix operator of homogeneity degree 4.
In particular, 
\begin{equation} \label{def mM}
\mM(u,v) = \mA[u,u,u,u] + \mB[u,u,u,v] 
+ \mC[u,u,v,v] + \mD[u,v,v,v] + \mF[v,v,v,v],
\end{equation}
and 
\[
\mA[u,u,u,u] = 
\begin{pmatrix} 
\mA_{11}[u,u,u,u] & \mA_{12}[u,u,u,u] \\ 
\mA_{21}[u,u,u,u] & \mA_{22}[u,u,u,u] 
\end{pmatrix}
\]
and similarly for the other terms and for $\mM(u,v)$. 
We assume that 
\begin{align*}
\mA[u^{(1)} , u^{(2)}, u^{(3)}, u^{(4)}]
& = \mA[u^{(2)} , u^{(1)}, u^{(3)}, u^{(4)}]
= \mA[u^{(1)} , u^{(2)}, u^{(4)}, u^{(3)}],
\\ 
\mB[u^{(1)} , u^{(2)}, u^{(3)}, v]
& = \mB[u^{(2)} , u^{(1)}, u^{(3)}, v],
\\
\mC[u^{(1)} , u^{(2)}, v^{(1)}, v^{(2)}]
& = \mC[u^{(2)} , u^{(1)}, v^{(1)}, v^{(2)}]
= \mC[u^{(1)} , u^{(2)}, v^{(2)}, v^{(1)}],
\\
\mD[u, v^{(1)}, v^{(2)}, v^{(3)}]
& = \mD[u, v^{(1)}, v^{(3)}, v^{(2)}],
\\
\mF[v^{(1)} , v^{(2)}, v^{(3)}, v^{(4)}]
& = \mF[v^{(2)} , v^{(1)}, v^{(3)}, v^{(4)}]
= \mF[v^{(1)} , v^{(2)}, v^{(4)}, v^{(3)}],
\end{align*}
for all $u,v,$ $u^{(n)}, v^{(n)}$, $n=1,2,3,4$.
We also assume that 
\[
\mC_{11}[u^{(1)} , u^{(2)}, v^{(1)}, v^{(2)}] h
= \sum_{j,\ell, k} u^{(1)}_j u^{(2)}_{-j} v^{(1)}_{\ell} v^{(2)}_{-\ell} h_k \, c_{11}(j,\ell,k)
\, e^{ik \cdot x}
\]
for some coefficient $c_{11}(j,\ell,k)$ to be determined, 
and similarly for all the other terms.
One has 
\[
\pa_t \begin{pmatrix} w \\ z \end{pmatrix}
= (I + \mM(u,v)) \begin{pmatrix} \pa_t u \\ \pa_t v \end{pmatrix} 
+ \{ \pa_t \mM(u,v) \} \begin{pmatrix} u \\ v \end{pmatrix} 
= (I + \mK(u,v)) \begin{pmatrix} \pa_t u \\ \pa_t v \end{pmatrix} 
\]
where 
\begin{equation} \label{def mK}
\mK(u,v) := (\Phi^{(5)})'(u,v) - I 
= \mM(u,v) + \mE(u,v)
\end{equation}
and 
\begin{align}
& \mE(u,v) \begin{pmatrix} \a \\ \b \end{pmatrix} 
:= \{ 
2 \mA[u, \a, u,u] 
+ 2 \mA[u,u,u,\a]
+ 2 \mB[u, \a, u,v]
\notag \\ & \quad \ 
+ \mB[u,u,\a,v]
+ \mB[u,u,u,\b]
+ 2 \mC[u, \a, v, v]
+ 2 \mC[u, u, v, \b]
+ \mD[\a, v, v, v]
\notag \\ & \quad \ 
+ \mD[u, \b, v, v]
+ 2 \mD[u, v, v, \b]
+ 2 \mF[v, \b, v, v]
+ 2 \mF[v, v, v, \b]
\} \begin{pmatrix} u \\ v \end{pmatrix}. 
\label{def mE}
\end{align}
The transformed equation is
\[
\pa_t \begin{pmatrix} u \\ v \end{pmatrix} = W(u,v)
\]
where 
\begin{equation} \label{def W}
W(u,v) := (I + \mK(u,v))^{-1} X^+( \Phi^{(5)}(u,v) ).
\end{equation}
Recalling \eqref{X+ new}, we decompose 
\begin{equation} \label{W decomp}
W(u,v) = \big( 1 + \mP(\Phi^{(5)}(u,v)) \big) \big( \mD_1(u,v) + X_3^+(u,v) \big) 
+ W_5(u,v) + W_{\geq 7}(u,v),
\end{equation}
where $(1 + \mP(\Phi^{(5)})) (\mD_1 + X_3^+)$ give no contribution to the energy estimate, 
\begin{equation} \label{def W5}
W_5(u,v) := X_5^+(u,v) + \mD_1(\mM(u,v)[u,v]) - \mK(u,v) \mD_1(u,v)
\end{equation}
and $W_{\geq 7}(u,v)$ is defined by difference and contains only terms 
of homogeneity at least seven in $(u,v)$.

We calculate each term of the first component $(W_5)_1$ of $W_5$. 
First, one has 
\begin{align*}
(W_5)_1 (u,v) 
& = (X_5^+)_1(u,v) - i \Lm  \big( \mM_{11}(u,v) u + \mM_{12}(u,v) v \big) 
\\ & \quad \ 
- \Big( \mM_{11}(u,v) (-i \Lm u) + \mM_{12}(u,v) (i \Lm v) \Big) 
- \Big( \mE(u,v) \binom{-i\Lm u}{i \Lm v} \Big)_1
\\
& = (X_5^+)_1(u,v) - 2 i \mM_{12}(u,v) \Lm v
- \Big( \mE(u,v) \binom{-i\Lm u}{i \Lm v} \Big)_1.
\end{align*}
Now
\begin{align*}
& \Big( \mE(u,v) \begin{pmatrix} -i\Lm u \\ i \Lm v \end{pmatrix} \Big)_1
= - 2 i \mA_{11}[u, \Lm u, u,u] u
- 2 i \mA_{11}[u,u,u,\Lm u] u
- 2 i \mB_{11}[u, \Lm u, u,v] u
\\ & \quad \ 
- i \mB_{11}[u,u, \Lm u,v] u
+ i \mB_{11}[u,u,u, \Lm v] u
- 2 i \mC_{11}[u, \Lm u, v, v] u
+ 2 i \mC_{11}[u, u, v, \Lm v] u
\\ & \quad \ 
- i \mD_{11}[\Lm u, v, v, v] u
+ i \mD_{11}[u, \Lm v, v, v] u
+ 2 i \mD_{11}[u, v, v, \Lm v] u
+ 2 i \mF_{11}[v, \Lm v, v, v] u
\\ & \quad \ 
+ 2 i \mF_{11}[v, v, v, \Lm v] u
- 2 i \mA_{12}[u, \Lm u, u,u] v
- 2 i \mA_{12}[u,u,u, \Lm u] v
- 2 i \mB_{12}[u, \Lm u, u,v] v
\\ & \quad \ 
- i \mB_{12}[u,u, \Lm u,v] v
+ i \mB_{12}[u,u,u, \Lm v] v
- 2 i \mC_{12}[u, \Lm u, v, v] v
+ 2 i \mC_{12}[u, u, v, \Lm v] v
\\ & \quad \ 
- i \mD_{12}[\Lm u, v, v, v] v
+ i \mD_{12}[u, \Lm v, v, v] v
+ 2 i \mD_{12}[u, v, v, \Lm v] v
+ 2 i \mF_{12}[v, \Lm v, v, v] v
\\ & \quad \ 
+ 2 i \mF_{12}[v, v, v, \Lm v] v.
\end{align*}
Thus the terms in $(W_5)_1(u,v)$ containing the monomials 
$u_j u_{-j} u_\ell u_{-\ell} u_k e^{ik \cdot x}$ are
\begin{align*}
& Y^{(4)}_{11} [u,u,u,u] u 
+ 2 i \mA_{11}[u, \Lm u, u,u] u 
+ 2 i \mA_{11}[u, u, u, \Lm u] u 
\\
& = \sum_{j,\ell,k} u_j u_{-j} u_\ell u_{-\ell} u_k e^{ik \cdot x} 
\Big( 2 i (|j| + |\ell|) a_{11}(j,\ell,k) 
- \frac{i}{64} \Big( \frac{|j|^2 |\ell|^2}{|j| + |k|} 
+ \frac{|j|^2 |\ell|^2}{|\ell| + |k|} \Big) \Big).
\end{align*}
Hence we choose
\begin{equation} \label{def a11 ottobrata}
a_{11}(j,\ell,k) := \frac{ |j|^2 |\ell|^2}{128 (|j| + |\ell|)}  
\Big( \frac{1}{|j| + |k|} + \frac{1}{|\ell| + |k|} \Big),
\end{equation}
so that $(W_5)_1(u,v)$ does not contain monomials of the type 
$u_j u_{-j} u_\ell u_{-\ell} u_k e^{ik \cdot x}$. 

Next, since $(X_5^+)_1(u,v)$ does not contain monomials 
$u_j u_{-j} u_\ell v_{-\ell} u_k e^{ik \cdot x}$, 
we fix 
\begin{equation} \label{def b11 ottobrata}
\mB_{11} = 0,
\end{equation}
so that $(W_5)_1(u,v)$ also does not contain such monomials.

Next, the terms in $(W_5)_1(u,v)$ containing the monomials 
$u_j u_{-j} v_\ell v_{-\ell} u_k e^{ik \cdot x}$ are
\begin{align*}
& Y^{(2)}_{11} [u,u,v,v] u
+ 2i \mC_{11} [u,\Lm u, v, v] u
- 2i \mC_{11} [u, u, v, \Lm v] u
\\
& = \sum_{j,\ell,k} u_j u_{-j} v_\ell v_{-\ell} u_k e^{ik \cdot x} 
\Big\{ \frac{i}{32} |j|^2 |\ell|^2 
\Big( \frac{- \d_{|\ell|}^{|k|} (1 - \d_{|j|}^{|k|})}{|j| - |k|}
+ \frac{1}{|j|+|k|} 
- \frac{(1 - \d_{|\ell|}^{|k|})}{|\ell| - |k|} \Big) 
\notag \\ & \qquad 
+ 2 i c_{11}(j,\ell,k) (|j| - |\ell|) \Big\}.
\end{align*}
This term can be eliminated for $|j| \neq |\ell|$, 
while for $|j| = |\ell|$ it cannot be eliminated, 
and in that case we fix $c_{11} = 0$. 
Thus we choose 
\begin{equation} \label{def c11 ottobrata}
c_{11}(j,\ell,k) := \frac{1}{64} |j|^2 |\ell|^2 
\Big( \frac{- \d_{|\ell|}^{|k|} (1 - \d_{|j|}^{|k|})}{|j| - |k|}
+ \frac{1}{|j|+|k|} 
- \frac{(1 - \d_{|\ell|}^{|k|})}{|\ell| - |k|} \Big)
\frac{1 - \d_{|j|}^{|\ell|} }{|\ell| - |j|},
\end{equation}
and the terms in $(W_5)_1(u,v)$ containing the monomials 
$u_j u_{-j} v_\ell v_{-\ell} u_k e^{ik \cdot x}$ become 
\begin{align*}
& \sum_{\begin{subarray}{c} j,\ell,k \\ |j| = |\ell| \end{subarray}} 
u_j u_{-j} v_\ell v_{-\ell} u_k e^{ik \cdot x} 
\Big\{ \frac{i}{32} |j|^2 |\ell|^2 
\Big( \frac{- \d_{|\ell|}^{|k|} (1 - \d_{|j|}^{|k|})}{|j| - |k|}
+ \frac{1}{|j|+|k|} 
- \frac{(1 - \d_{|\ell|}^{|k|})}{|\ell| - |k|} \Big) \Big\}
\\
& = \frac{i}{32} \sum_{\begin{subarray}{c} j,\ell,k \\ |j| = |\ell| \end{subarray}} 
u_j u_{-j} v_\ell v_{-\ell} u_k e^{ik \cdot x} 
|j|^2 |\ell|^2 
\Big( \frac{1}{|j|+|k|} - \frac{(1 - \d_{|\ell|}^{|k|})}{|\ell| - |k|} \Big).
\end{align*}

Next, since $(X_5^+)_1(u,v)$ does not contain monomials 
$u_j v_{-j} v_\ell v_{-\ell} u_k e^{ik \cdot x}$, 
we fix 
\begin{equation} \label{def d11 ottobrata}
\mD_{11} = 0,
\end{equation}
so that $(W_5)_1(u,v)$ also does not contain such monomials.

Next, the terms in $(W_5)_1(u,v)$ containing the monomials 
$v_j v_{-j} v_\ell v_{-\ell} u_k e^{ik \cdot x}$ are
\begin{align*}
& Y^{(0)}_{11} [v,v,v,v] u
- 2i \mF_{11} [v,\Lm v, v, v] u
- 2i \mF_{11} [v, v, v, \Lm v] u
\\
& = \sum_{j,\ell,k} v_j v_{-j} v_\ell v_{-\ell} u_k e^{ik \cdot x} 
\Big\{ \frac{i}{64} |j|^2 |\ell|^2 
\Big( - \frac{\d_{|\ell|}^{|k|} + \d_{|j|}^{|k|}}{|j| + |\ell|}
+ \frac{(1 - \d_{|j|}^{|k|})}{|j| - |k|} 
+ \frac{(1 - \d_{|\ell|}^{|k|})}{|\ell| - |k|} \Big) 
\notag \\ & \quad \ 
- 2 i f_{11}(j,\ell,k) (|j| + |\ell|) \Big\}.
\end{align*}
Hence we fix 
\begin{equation} \label{def f11 ottobrata}
f_{11}(j,\ell,k) 
:= \frac{1}{128} 
\Big( - \frac{\d_{|\ell|}^{|k|} + \d_{|j|}^{|k|}}{|j| + |\ell|}
+ \frac{(1 - \d_{|j|}^{|k|})}{|j| - |k|} 
+ \frac{(1 - \d_{|\ell|}^{|k|})}{|\ell| - |k|} \Big) \frac{|j|^2 |\ell|^2 }{|j|+|\ell|},
\end{equation}
so that $(W_5)_1(u,v)$ does not contain monomials of the type 
$v_j v_{-j} v_\ell v_{-\ell} u_k e^{ik \cdot x}$. 

Next, the terms in $(W_5)_1(u,v)$ containing the monomials 
$u_j u_{-j} u_\ell u_{-\ell} v_k e^{ik \cdot x}$ are
\begin{align*}
& Y^{(4)}_{12} [u,u,u,u] v
- 2i \mA_{12} [u,u,u,u] \Lm v
+ 2i \mA_{12} [u,\Lm u, u,u] v
+ 2i \mA_{12} [u,u,u,\Lm u] v
\\
& = \sum_{j,\ell,k} u_j u_{-j} u_\ell u_{-\ell} v_k e^{ik \cdot x} 
\Big\{ \frac{3i}{32} |j| |\ell| (|j| + |\ell|)
- 2 i a_{12}(j,\ell,k) (|k| - |j| - |\ell|) \Big\}.
\end{align*}
Hence we fix 
\begin{equation} \label{def a12 ottobrata}
a_{12}(j,\ell,k) := \frac{3}{64} |j| |\ell| (|j| + |\ell|)
\frac{(1 - \d_{|k|}^{|j|+|\ell|} )}{|k| - |j| - |\ell|} ,
\end{equation}
and the terms in $(W_5)_1(u,v)$ containing the monomials 
$u_j u_{-j} u_\ell u_{-\ell} v_k e^{ik \cdot x}$ become 
\begin{align*}
& \frac{3 i}{32} 
\sum_{\begin{subarray}{c} j,\ell,k \\ |k| = |j| + |\ell| \end{subarray}} 
u_j u_{-j} u_\ell u_{-\ell} v_k e^{ik \cdot x} 
|j| |\ell| |k|. 
\end{align*}

Next, the terms in $(W_5)_1(u,v)$ containing the monomials 
$u_j u_{-j} u_\ell v_{-\ell} v_k e^{ik \cdot x}$ are
\begin{align*}
& Y^{(3)}_{12} [u,u,u,v] v
- 2i \mB_{12} [u,u,u,v] \Lm v
+ 2i \mB_{12} [u,\Lm u, u,v] v
\\ 
& \qquad + i \mB_{12} [u,u,\Lm u,v] v
- i \mB_{12} [u,u,u,\Lm v] v
\\
& = \sum_{j,\ell,k} u_j u_{-j} u_\ell v_{-\ell} v_k e^{ik \cdot x} 
\Big\{ \frac{i}{16} |j|^2 |\ell|
\Big( \frac{|\ell| \d_{|\ell|}^{|j|} (1 - \d_{|\ell|}^{|k|}) }{|\ell| - |k|}
+ 6 
+ \frac{|\ell|}{|\ell| + |j|}
\notag \\ & \qquad 
+ \frac{|\ell| (1 - \d_{|\ell|}^{|j|}) }{|\ell| - |j|} \Big) 
- 2 i b_{12}(j,\ell,k) (|k| - |j|) \Big\}.
\end{align*}
Hence we fix 
\begin{equation} \label{def b12 ottobrata}
b_{12}(j,\ell,k) := \frac{|j|^2 |\ell|}{32} 
\Big( \frac{|\ell| \d_{|\ell|}^{|j|} (1 - \d_{|\ell|}^{|k|}) }{|\ell| - |k|}
+ 6
+ \frac{|\ell|}{|\ell| + |j|}
+ \frac{|\ell| (1 - \d_{|\ell|}^{|j|}) }{|\ell| - |j|} \Big) 
\frac{1 - \d_{|j|}^{|k|} }{|k| - |j|},
\end{equation}
and the terms in $(W_5)_1(u,v)$ containing the monomials 
$u_j u_{-j} u_\ell v_{-\ell} v_k e^{ik \cdot x}$ become 
\begin{align*}
& \sum_{\begin{subarray}{c} j,\ell,k \\ |j| = |k| \end{subarray}} 
u_j u_{-j} u_\ell v_{-\ell} v_k e^{ik \cdot x} 
\frac{i}{16} |j|^2 |\ell|
\Big( \frac{|\ell| \d_{|\ell|}^{|j|} (1 - \d_{|\ell|}^{|k|}) }{|\ell| - |k|}
+ 6
+ \frac{|\ell|}{|\ell| + |j|}
+ \frac{|\ell| (1 - \d_{|\ell|}^{|j|}) }{|\ell| - |j|} \Big) 
\\ 
& = \frac{i}{16}  \sum_{\begin{subarray}{c} j,\ell,k \\ |j| = |k| \end{subarray}} 
u_j u_{-j} u_\ell v_{-\ell} v_k e^{ik \cdot x} 
|j|^2 |\ell|
\Big( 6 + \frac{|\ell|}{|\ell| + |j|}
+ \frac{|\ell| (1 - \d_{|\ell|}^{|j|}) }{|\ell| - |j|} \Big).
\end{align*}

Next, the terms in $(W_5)_1(u,v)$ containing the monomials 
$u_j u_{-j} v_\ell v_{-\ell} v_k e^{ik \cdot x}$ are
\begin{align*}
& Y^{(2)}_{12} [u,u,v,v] v
- 2i \mC_{12} [u,u,v,v] \Lm v
+ 2i \mC_{12} [u,\Lm u, v,v] v
- 2i \mC_{12} [u,u,v,\Lm v] v
\\
& = \sum_{j,\ell,k} u_j u_{-j} v_\ell v_{-\ell} v_k e^{ik \cdot x} 
\Big\{ 
\frac{3i}{16} |j| |\ell| (|j| - |\ell|)  
- 2 i c_{12}(j,\ell,k) (|k| - |j| + |\ell|) \Big\}.
\end{align*}
Hence we fix
\begin{equation} \label{def c12 ottobrata}
c_{12}(j,\ell,k) := \frac{3}{32} |j| |\ell| (|j| - |\ell|) 
\frac{1 - \d_{|k|}^{|j| - |\ell|} }{|k| - |j| + |\ell|},
\end{equation}
and the terms in $(W_5)_1(u,v)$ containing the monomials 
$u_j u_{-j} v_\ell v_{-\ell} v_k e^{ik \cdot x}$ become
\begin{align*}
\frac{3 i}{16} \sum_{\begin{subarray}{c} j,\ell,k \\ |k| = |j| - |\ell| \end{subarray}} 
u_j u_{-j} v_\ell v_{-\ell} v_k e^{ik \cdot x} |j| |\ell| |k|.
\end{align*}

Next, the terms in $(W_5)_1(u,v)$ containing the monomials 
$u_j v_{-j} v_\ell v_{-\ell} v_k e^{ik \cdot x}$ are
\begin{align*}
& Y^{(1)}_{12} [u,v,v,v] v
- 2i \mD_{12} [u,v,v,v] \Lm v
+ i \mD_{12} [\Lm u, v,v,v] v
\\ & \qquad 
- i \mD_{12} [u, \Lm v, v,v] v
- 2i \mD_{12} [u,v,v,\Lm v] v
\\
& = \sum_{j,\ell,k} u_j v_{-j} v_\ell v_{-\ell} v_k e^{ik \cdot x} 
\Big\{ \frac{i}{16} \sum_{j,\ell,k} |j| |\ell|^2
\Big( \frac{- |j| \d_{|j|}^{|\ell|} }{|j| + |k|} 
- 6 + \frac{|j| (1 - \d_{|j|}^{|\ell|}) }{|\ell| - |j|}
\notag \\ & \qquad 
- \frac{|j|}{|\ell| + |j|} \Big) 
- 2 i d_{12}(j,\ell,k) (|k| + |\ell|) \Big\}.
\end{align*}
Hence we fix
\begin{equation} \label{def d12 ottobrata}
d_{12}(j,\ell,k) := \frac{|j| |\ell|^2}{32 (|k| + |\ell|)}
\Big( \frac{- |j| \d_{|j|}^{|\ell|} }{|j| + |k|} 
- 6 + \frac{|j| (1 - \d_{|j|}^{|\ell|}) }{|\ell| - |j|}
- \frac{|j|}{|\ell| + |j|} \Big),
\end{equation}
so that $(W_5)_1(u,v)$ does not contain monomials of the type 
$u_j v_{-j} v_\ell v_{-\ell} v_k e^{ik \cdot x}$. 

Next, the terms in $(W_5)_1(u,v)$ containing the monomials 
$v_j v_{-j} v_\ell v_{-\ell} v_k e^{ik \cdot x}$ are
\begin{align*}
& Y^{(0)}_{12} [v,v,v,v] v
- 2i \mF_{12} [v,v,v,v] \Lm v
- 2i \mF_{12} [v, \Lm v,v,v] v
- 2i \mF_{12} [v,v,v, \Lm v] v
\\
& = \sum_{j,\ell,k} v_j v_{-j} v_\ell v_{-\ell} v_k e^{ik \cdot x} 
\Big\{ - \frac{3i}{32} \sum_{j,\ell,k} |j| |\ell| (|j| + |\ell|) 
- 2i f_{12}(j,\ell,k) (|k| + |j| + |\ell|) \Big\}.
\end{align*}
Hence we fix 
\begin{equation} \label{def f12 ottobrata}
f_{12}(j,\ell,k) := - \frac{3 |j| |\ell| (|j| + |\ell|) }{64 (|k| + |j| + |\ell|)},
\end{equation}
so that $(W_5)_1(u,v)$ does not contain monomials of the type 
$v_j v_{-j} v_\ell v_{-\ell} v_k e^{ik \cdot x}$. 

Summarizing, it remains 
\begin{align}
(W_5)_1(u,v) 
& = \frac{i}{32} \sum_{\begin{subarray}{c} j,\ell,k \\ |j| = |\ell| \end{subarray}} 
u_j u_{-j} v_\ell v_{-\ell} u_k e^{ik \cdot x} 
|j|^2 |\ell|^2 
\Big( \frac{1}{|j|+|k|} - \frac{(1 - \d_{|\ell|}^{|k|})}{|\ell| - |k|} \Big)
\notag \\ & \quad 
+ \frac{3 i}{32} 
\sum_{\begin{subarray}{c} j,\ell,k \\ |k| = |j| + |\ell| \end{subarray}} 
u_j u_{-j} u_\ell u_{-\ell} v_k e^{ik \cdot x} 
|j| |\ell| |k|
\notag \\ & \quad 
+ \frac{i}{16}  \sum_{\begin{subarray}{c} j,\ell,k \\ |j| = |k| \end{subarray}} 
u_j u_{-j} u_\ell v_{-\ell} v_k e^{ik \cdot x} 
|j|^2 |\ell|
\Big( 6 + \frac{|\ell|}{|\ell| + |j|}
+ \frac{|\ell| (1 - \d_{|\ell|}^{|j|}) }{|\ell| - |j|} \Big)
\notag \\ & \quad 
+ \frac{3 i}{16} \sum_{\begin{subarray}{c} j,\ell,k \\ |k| = |j| - |\ell| \end{subarray}} 
u_j u_{-j} v_\ell v_{-\ell} v_k e^{ik \cdot x} |j| |\ell| |k|.
\label{W5 comp 1 after NF}
\end{align}

With similar calculations, 
or deducing the formula from the real structure, 
the second component $(W_5)_2$ of $W_5$ is 
\begin{align}
(W_5)_2(u,v) 
& = - \frac{i}{32} \sum_{\begin{subarray}{c} j,\ell,k \\ |j| = |\ell| \end{subarray}} 
v_j v_{-j} u_\ell u_{-\ell} v_k e^{ik \cdot x} 
|j|^2 |\ell|^2 
\Big( \frac{1}{|j|+|k|} - \frac{(1 - \d_{|\ell|}^{|k|})}{|\ell| - |k|} \Big)
\notag \\ & \quad 
- \frac{3 i}{32} 
\sum_{\begin{subarray}{c} j,\ell,k \\ |k| = |j| + |\ell| \end{subarray}} 
v_j v_{-j} v_\ell v_{-\ell} u_k e^{ik \cdot x} 
|j| |\ell| |k|
\notag \\ & \quad 
- \frac{i}{16}  \sum_{\begin{subarray}{c} j,\ell,k \\ |j| = |k| \end{subarray}} 
v_j v_{-j} v_\ell u_{-\ell} u_k e^{ik \cdot x} 
|j|^2 |\ell|
\Big( 6 + \frac{|\ell|}{|\ell| + |j|}
+ \frac{|\ell| (1 - \d_{|\ell|}^{|j|}) }{|\ell| - |j|} \Big)
\notag \\ & \quad 
- \frac{3 i}{16} \sum_{\begin{subarray}{c} j,\ell,k \\ |k| = |j| - |\ell| \end{subarray}} 
v_j v_{-j} u_\ell u_{-\ell} u_k e^{ik \cdot x} |j| |\ell| |k|.
\label{W5 comp 2 after NF}
\end{align}

\begin{lemma} \label{lemma:stima W5} 
For all $s \geq 0$, all $(w,z) \in H^s_0(\T^d, c.c.) \cap 
H^{m_0}_0(\T^d, c.c.)$, one has 
\begin{equation*} 
\| W_{5}(u,v) \|_s 
\leq C \| u \|_{m_0}^4 \| u \|_s,
\end{equation*}
where $C$ is a universal constant.
\end{lemma}

\begin{proof} 
The estimate is deduced from \eqref{W5 comp 1 after NF}-\eqref{W5 comp 2 after NF},
using the following bound: 
if $\a, \b \in \Z^d \setminus \{ 0 \}$, $0 < ||\a| - |\b|| < 1$, 
then $|\a|^2 - |\b|^2$ is a nonzero integer, 
$|\a| \leq 2 |\b|$, $|\b| \leq 2 |\a|$, and 
\begin{equation} \label{denom due interi}
\frac{1}{||\a| - |\b||} = \frac{|\a| + |\b|}{| |\a|^2 - |\b|^2 |} 
\leq |\a| + |\b| 
\leq C |\a| 
\leq C' |\b|. 
\end{equation}
\end{proof}

By \eqref{W5 comp 1 after NF}-\eqref{W5 comp 2 after NF}, 
the system for the Fourier coefficients becomes
\begin{align}
\pa_t u_k 
& = - i (1 + \mP) \Big( |k| u_k 
+ \frac{1}{4} \sum_{|j| = |k|} u_j u_{-j} |j|^2 v_k \Big)
\notag \\ & \quad 
+ \frac{i}{32} \sum_{\begin{subarray}{c} j,\ell \\ |j| = |\ell| \end{subarray}} 
u_j u_{-j} v_\ell v_{-\ell} u_k |j|^2 |\ell|^2 
\Big( \frac{1}{|j|+|k|} - \frac{(1 - \d_{|\ell|}^{|k|})}{|\ell| - |k|} \Big)
\notag \\ & \quad 
+ \frac{3 i}{32} 
\sum_{\begin{subarray}{c} j,\ell \\ |j| + |\ell| = |k| \end{subarray}} 
u_j u_{-j} u_\ell u_{-\ell} v_k |j| |\ell| |k|
\notag \\ & \quad 
+ \frac{i}{16} \sum_{\begin{subarray}{c} j,\ell \\ |j| = |k| \end{subarray}} 
u_j u_{-j} u_\ell v_{-\ell} v_k |j|^2 |\ell|
\Big( 6 + \frac{|\ell|}{|\ell| + |j|}
+ \frac{|\ell| (1 - \d_{|\ell|}^{|j|}) }{|\ell| - |j|} \Big)
\notag \\ & \quad 
+ \frac{3 i}{16} \sum_{\begin{subarray}{c} j,\ell \\ |j| - |\ell| = |k| \end{subarray}} 
u_j u_{-j} v_\ell v_{-\ell} v_k |j| |\ell| |k| 
+ [(W_{\geq 7})_1(u,v)]_k
\label{eq for uk}
\end{align}
and
\begin{align}
\pa_t v_k 
& = i (1 + \mP) \Big( |k| v_k 
+ \frac{1}{4} \sum_{|j| = |k|} v_j v_{-j} |j|^2 u_k \Big)
\notag \\ & \quad 
- \frac{i}{32} \sum_{\begin{subarray}{c} j,\ell \\ |j| = |\ell| \end{subarray}} 
v_j v_{-j} u_\ell u_{-\ell} v_k |j|^2 |\ell|^2 
\Big( \frac{1}{|j|+|k|} - \frac{(1 - \d_{|\ell|}^{|k|})}{|\ell| - |k|} \Big)
\notag \\ & \quad 
- \frac{3 i}{32} 
\sum_{\begin{subarray}{c} j,\ell \\ |j| + |\ell| = |k| \end{subarray}} 
v_j v_{-j} v_\ell v_{-\ell} u_k |j| |\ell| |k|
\notag \\ & \quad 
- \frac{i}{16} \sum_{\begin{subarray}{c} j,\ell \\ |j| = |k| \end{subarray}} 
v_j v_{-j} v_\ell u_{-\ell} u_k |j|^2 |\ell|
\Big( 6 + \frac{|\ell|}{|\ell| + |j|}
+ \frac{|\ell| (1 - \d_{|\ell|}^{|j|}) }{|\ell| - |j|} \Big)
\notag \\ & \quad 
- \frac{3 i}{16} \sum_{\begin{subarray}{c} j,\ell \\ |j| - |\ell| = |k| \end{subarray}} 
v_j v_{-j} u_\ell u_{-\ell} u_k |j| |\ell| |k| 
+ [(W_{\geq 7})_2(u,v)]_k
\label{eq for vk}
\end{align}
where $[(W_{\geq 7})_1(u,v)]_k$ denotes the $k$-th Fourier coefficient 
of the first component of $W_{\geq 7}(u,v)$, and similarly for the second component.

Now we prove that the transformation $\Phi^{(5)}$ is bounded and invertible in a ball.
Let us begin with estimating the denominators $|k| \pm |j| \pm |\ell|$. 

\begin{lemma} \label{lemma:ottobrata.1}
Let $d \geq 2$, and let $k, j, \ell \in \Z^d \setminus \{ 0 \}$. 
If $|k| - |j| + |\ell|$ is nonzero, then 
\begin{equation} \label{ott.1}
\Big| \frac{1}{|k| - |j| + |\ell|} \Big| \leq C |j|^2 |\ell|.
\end{equation}
If $|k| - |j| - |\ell|$ is nonzero, then 
\begin{equation} \label{ott.4}
\Big| \frac{1}{|k| - |j| - |\ell|} \Big| \leq C |j| |\ell| (|j| + |\ell|).
\end{equation}
The constant $C$ is universal 
($C = 27$ is enough). 
\end{lemma}

\begin{proof} 
Let $|k| - |j| + |\ell| \neq 0$. 
If $||k| - |j| + |\ell|| \geq 1$, then \eqref{ott.1} trivially holds. 
Thus, assume that 
\begin{equation} \label{ott.3}
0 < ||k| - |j| + |\ell|| < 1.
\end{equation} 
Since $|j| \geq 1$, it follows that 
\begin{equation} \label{ott.8}
|k| + |\ell| < |j| + 1 
\leq 2|j|.
\end{equation}
The product 
\begin{align}
p & := 
(|k| + |j| + |\ell|) 
(|k| + |j| - |\ell|) 
(|k| - |j| + |\ell|) 
(|k| - |j| - |\ell|) 
\notag \\ & 
= (|k|^2 + |j|^2 - |\ell|^2)^2 - 4 |k|^2 |j|^2 
\label{ott.6}
\end{align}
is an integer. 
If $p \neq 0$, then $|p| \geq 1$, and, using \eqref{ott.8}, 
\begin{align*}
\Big| \frac{1}{|k| - |j| + |\ell|} \Big| 
& \leq |(|k| + |j| + |\ell|) (|k| + |j| - |\ell|) (|k| - |j| - |\ell|)|
\\ 
& \leq (3 |j|) (3|j|) (3|\ell|) = C |j|^2 |\ell|.
\end{align*}
If $p=0$, then $|k| + |j| - |\ell| = 0$ or $|k| - |j| - |\ell| = 0$. 
If $|k| + |j| - |\ell| = 0$, then $|k| - |j| + |\ell| = 2 |k| \geq 2$, 
which contradicts \eqref{ott.3}. 
If $|k| - |j| - |\ell| = 0$, then $|k| - |j| + |\ell| = 2 |\ell| \geq 2$, 
which also contradicts \eqref{ott.3}. 
This completes the proof of \eqref{ott.1}. 

Now we prove \eqref{ott.4}. 
Let $|k| - |j| - |\ell| \neq 0$. 
If $||k| - |j| - |\ell|| \geq 1$, then \eqref{ott.4} trivially holds. 
Thus, assume that 
\begin{equation} \label{ott.5}
0 < ||k| - |j| - |\ell|| < 1.
\end{equation} 
Then 
\[
|k| < |j| + |\ell| + 1 \leq 2 (|j| + |\ell|).
\]
Recalling \eqref{ott.6}, if $p \neq 0$, then $|p| \geq 1$, and 
\begin{align*}
\Big| \frac{1}{|k| - |j| - |\ell|} \Big| 
& \leq |(|k| + |j| + |\ell|) (|k| + |j| - |\ell|) (|k| - |j| + |\ell|)| 
\leq C (|j|+|\ell|) |j| |\ell|.
\end{align*}
If $p=0$, then $|k| + |j| - |\ell| = 0$ or $|k| - |j| + |\ell| = 0$. 
If $|k| + |j| - |\ell| = 0$, then $||k| - |j| - |\ell|| = 2 |j| \geq 2$, 
which contradicts \eqref{ott.5}. 
If $|k| - |j| + |\ell| = 0$, then $||k| - |j| - |\ell|| = 2 |\ell| \geq 2$, 
which also contradicts \eqref{ott.5}. 
\end{proof}

\begin{remark} \label{rem:Oppa! Ema's style}
The bound $|p| \geq 1$ in the proof of Lemma \ref{lemma:ottobrata.1} is sharp. 
Indeed, it is enough to show that there are infinitely many choices of $k,j,\ell\in\Z^d\setminus\{0\}$ such that the triple $(|k|^2,|j|^2,|\ell|^2)$ is of the form $(n,n+1,4n+2)$ for some $n\in\N$.
In dimension $d \geq 3$, this is trivial. 

In dimension $d = 2$, recall that the set of integers that can be written as the sum of two squares is closed under multiplication, by Brahmagupta's identity 
\[
(x^2+y^2)(z^2+w^2)=(xz+yw)^2 + (xw-yz)^2. 
\]
Then, it is enough to observe that for $n=4$ the triple $(n,n+1,4n+2)=(4,5,18)=(2^2+0^2,2^2+1^2,3^2+3^2)$ contains only numbers that are the sum of two squares, and that, given any triple $(n,n+1,4n+2)$ that contains only numbers that are the sum of two squares, the triple $(2n^2 + 2n, 2n^2 + 2n +1, 4(2n^2+2n)+2)$ has the same property. Indeed, $2n^2 + 2n +1 = n^2 + (n+1)^2$ and $4(2n^2+2n)+2 = (2n+1)^2 + (2n+1)^2$ are sums of two squares for any $n\in\N$, while $2n^2 + 2n = 2n(n+1)$ is the sum of two squares since it is the product of numbers that are the sum of two squares ($n,n+1$ are sums of two squares by assumption, and $2 = 1^2+1^2$). 
\end{remark}

\begin{lemma} \label{lemma:ottobrata.4} 
For $d \geq 2$, 
the coefficients $a_{11}, c_{11}, f_{11}, a_{12}, b_{12}, c_{12}, d_{12}, f_{12}$ 
in \eqref{def a11 ottobrata}-\eqref{def f12 ottobrata} 
all satisfy the bound 
\[
| \text{coefficient}(k,j,\ell)| \leq C (|j|^4 |\ell|^2 + |j|^2 |\ell|^4)
\]
for some universal constant $C$. 
For $d=1$, they satisfy 
\[
| \text{coefficient}(k,j,\ell)| \leq C |j|^2 |\ell|^2.
\]
\end{lemma}

\begin{proof} 
Let $d \geq 2$. 
The denominators estimated in Lemma \ref{lemma:ottobrata.1} 
appear only in $a_{12}$ and $c_{12}$. 
The estimate for $|a_{12}|$ directly follows from \eqref{def a12 ottobrata} 
and \eqref{ott.4}. 
To estimate $|c_{12}|$, for $0 < | |k| - |j| + |\ell|| < 1$ use \eqref{ott.1} 
and \eqref{ott.8}, otherwise $|c_{12}| \leq C |j| |\ell| (|j| + |\ell|)$. 
The estimate of $a_{11}, f_{12}$ is trivial.
To estimate $c_{11}, f_{11}, b_{12}, d_{12}$, use repeatedly 
bound \eqref{denom due interi}.
In dimension $d=1$ all the estimates are trivial.
\end{proof}

\begin{lemma} \label{lemma:ottobrata.5}
Let 
\begin{equation} \label{def m1}
m_1 := \begin{cases} 1 & \quad \text{if} \ d = 1, \\ 
2 & \quad \text{if} \ d \geq 2. 
\end{cases} 
\end{equation}
All the operators $\mG \in \{ \mA_{11}, \mC_{11}, \mF_{11}, 
\mA_{12}, \mB_{12}, \mC_{12}, \mD_{12}, \mF_{12} \}$ satisfy 
\begin{equation}
\| \mG[u, v, w, z] h \|_s 
\leq C \| u \|_{m_1} \| v \|_{m_1} \| w \|_{m_1} \| z \|_{m_1} \| h \|_s
\end{equation} 
for all complex functions $u,v,w,z,h$, all real $s \geq 0$, 
where $C$ is a universal constant. 
\end{lemma}

\begin{proof}
It is an immediate consequence of Lemma \ref{lemma:ottobrata.4}. 
\end{proof}

We recall the definition $\| (w,z) \|_s := \| w \|_s = \| z \|_s$ 
for all pairs $(w,z) = (w, \overline{w}) \in H^s_0(\T^d,c.c.)$ 
of complex conjugate functions.  
By \eqref{def mM}, \eqref{def mK}, \eqref{def mE}, we deduce the following estimates. 
 
\begin{lemma}
\label{lemma:stime mM mK} 
For all $s \geq 0$, all $(u,v) \in H^{m_1}_0(\T^d, c.c.)$, 
$(\a,\b) \in H^s_0(\T^d,c.c.)$ one has
\begin{align} \label{stima mM}
\Big\| \mM(u,v) \begin{pmatrix} \a \\ \b \end{pmatrix} \Big\|_s 
& \leq C \| u \|_{m_1}^4 \|\a \|_s ,
\\
\Big\| \mK(u,v) \begin{pmatrix} \a \\ \b \end{pmatrix} \Big\|_s 
& \leq C \| u \|_{m_1}^3 ( \| u \|_{m_1} \|\a \|_s + \| u \|_s \| \a \|_{m_1} ),
\label{stima mK}
\end{align}
where $m_1$ is defined in \eqref{def m1} and $C$ is a universal constant.
There exists a universal $\d > 0$ such that, for $\| u \|_{m_1} < \d$, 
the operator $(I + \mK(u,v)) : H^{m_1}_0(\T^d, c.c.)$ $\to H^{m_1}_0(\T^d, c.c.)$ 
is invertible, with inverse 
\begin{equation} \label{def tilde mK}
(I + \mK(u,v))^{-1} = I - \mK(u,v) + \tilde \mK(u,v), \quad 
\tilde \mK(u,v) := \sum_{n=2}^\infty (- \mK(u,v))^n,
\end{equation}
satisfying
\begin{equation*} 
\Big\| (I + \mK(u,v))^{-1} \begin{pmatrix} \a \\ \b \end{pmatrix} \Big\|_s 
\leq C (\| \a \|_s + \| u \|_{m_1}^3 \| u \|_s \| \a \|_{m_1}),
\end{equation*}
for all $s \geq 0$. 
\end{lemma}

The nonlinear, continuous map $\Phi^{(5)}$ is invertible in a ball around the origin. 

\begin{lemma}
\label{lemma:Phi5 inv}
There exists a universal constant $\d > 0$ such that, 
for all $(w,z) \in H^{m_1}_0(\T^d, c.c.)$ in the ball $\| w \|_{m_1} \leq \d$, 
there exists a unique $(u,v) \in H^{m_1}_0(\T^d, c.c.)$ such that 
$\Phi^{(5)}(u,v) = (w,z)$, with $\| u \|_{m_1} \leq 2 \| w \|_{m_1}$. 
If, in addition, $w \in H^s_0$ for some $s > m_1$, then $u$ also belongs to $H^s_0$, 
and $\| u \|_s \leq 2 \| w \|_s$.
This defines the continuous inverse map $(\Phi^{(5)})^{-1} : H^s_0(\T^d, c.c.) \cap 
\{ \| w \|_{m_1} \leq \d \}$ $\to H^s_0(\T^d, c.c.)$.
\end{lemma}

\begin{proof} 
Using the estimates of Lemma \ref{lemma:stime mM mK}, 
the proof of Lemma \ref{lemma:Phi5 inv} 
is a straightforward adaptation of the proof of Lemma 4.3 in \cite{Kold}.
\end{proof}

We estimate the remainder $W_{\geq 7}(u,v)$. 
By \eqref{W decomp} (which is the definition of $W_{\geq 7}(u,v)$) and
\eqref{def W5}, 
\eqref{def W}, 
\eqref{X+ new},
\eqref{def Phi5}, 
we calculate
\begin{align} 
W_{\geq 7}(u,v) 
& = \tilde \mK(u,v) [1 + \mP(\Phi^{(5)}(u,v))] \mD_1(u,v) 
\notag \\ & \quad 
+ (- \mK(u,v) + \tilde \mK(u,v)) [1 + \mP(\Phi^{(5)}(u,v))] \mD_1 (\mM(u,v)[u,v])
\notag \\ & \quad 
- \mK(u,v) \mP(\Phi^{(5)}(u,v)) \mD_1(u,v) 
\notag \\ & \quad 
+ \mP(\Phi^{(5)}(u,v)) \mD_1 (\mM(u,v)[u,v])
\notag \\ & \quad 
+ (- \mK(u,v) + \tilde \mK(u,v)) [1 + \mP(\Phi^{(5)}(u,v))] X_3^+(u,v)
\notag \\ & \quad 
+ (- \mK(u,v) + \tilde \mK(u,v)) X_5^+(u,v)
\notag \\ & \quad 
+ (I + \mK(u,v))^{-1} [1 + \mP(\Phi^{(5)}(u,v))] [X_3^+(\Phi^{(5)}(u,v)) - X_3^+(u,v)]
\notag \\ & \quad 
+ (I + \mK(u,v))^{-1} [X_5^+(\Phi^{(5)}(u,v)) - X_5^+(u,v)]
\notag \\ & \quad 
+ (I + \mK(u,v))^{-1} X_{\geq 7}^+(\Phi^{(5)}(u,v)),
\label{W geq 7}
\end{align}
where $\tilde \mK(u,v)$ is defined in \eqref{def tilde mK}.
The only unbounded operator appearing in \eqref{W geq 7} is $\mD_1$. 
We rewrite the terms containing $\mD_1$ by using
the ``homological equation'' \eqref{def W5}
(which is, in short, $\mD_1 \mM - \mK \mD_1 = W_5 - X_5^+$)
and the fact that the multiplication by $\mP(\Phi^{(5)}(u,v))$ commutes with $\mK(u,v)$, 
because $\mP(\Phi^{(5)}(u,v))$ is a real scalar function of time only.  
Thus, omitting to write $(u,v)$ everywhere, 
the first two terms in \eqref{W geq 7} become
\begin{align*}
& \tilde \mK (1 + \mP(\Phi^{(5)})) \mD_1
+ (- \mK + \tilde \mK) (1 + \mP(\Phi^{(5)})) \mD_1 \mM
\\ & \quad 
= (1 + \mP(\Phi^{(5)}))
\Big( \sum_{n=2}^\infty (-\mK)^n \mD_1 
+ \sum_{n=1}^\infty (-\mK)^n \mD_1 \mM \Big)
\\ & \quad 
= (1 + \mP(\Phi^{(5)})) \sum_{n=1}^\infty (-\mK)^n (-\mK \mD_1 + \mD_1 \mM)
\\ & \quad 
= (1 + \mP(\Phi^{(5)})) (-\mK + \tilde\mK) (W_5 - X_5^+). 
\end{align*}
Therefore \eqref{W geq 7} becomes
\begin{align} 
W_{\geq 7}(u,v) 
& = [1 + \mP(\Phi^{(5)}(u,v))] (-\mK(u,v) + \tilde\mK(u,v)) (W_5(u,v) - X_5^+(u,v))
\notag \\ & \quad 
+ \mP(\Phi^{(5)}(u,v)) (W_5(u,v) - X_5^+(u,v))
\notag \\ & \quad 
+ (- \mK(u,v) + \tilde \mK(u,v)) [1 + \mP(\Phi^{(5)}(u,v))] X_3^+(u,v) 
\notag \\ & \quad 
+ (- \mK(u,v) + \tilde \mK(u,v)) X_5^+(u,v) 
\notag \\ & \quad 
+ (I + \mK(u,v))^{-1} [1 + \mP(\Phi^{(5)}(u,v))] [X_3^+(\Phi^{(5)}(u,v)) - X_3^+(u,v)]
\notag \\ & \quad 
+ (I + \mK(u,v))^{-1} [X_5^+(\Phi^{(5)}(u,v)) - X_5^+(u,v)]
\notag \\ & \quad 
+ (I + \mK(u,v))^{-1} X_{\geq 7}^+(\Phi^{(5)}(u,v)).
\label{W geq 7 bis}
\end{align}

\begin{lemma} \label{lemma:stime W geq 7}
There exist universal constants $\d > 0$, $C > 0$ such that, 
for all $s \geq 0$, 
for all $(u,v) \in H^{m_1}_0(\T^d, c.c.) \cap H^s_0(\T^d, c.c.)$ 
in the ball $\| u \|_{m_1} \leq \d$, one has 
\begin{align*}
& \| W_{\geq 7}(u,v) \|_s \leq C \| u \|_{m_1}^6 \| u \|_s.
\end{align*}
\end{lemma}

\begin{proof} Use formula \eqref{W geq 7 bis} 
and Lemmas 
\ref{lemma:elementary},
\ref{lemma:20.05.2020}, 
\ref{lemma:stima W5},
\ref{lemma:stime mM mK},
\ref{lemma:Phi5 inv}.
\end{proof}

\bigskip

\noindent
\emph{Energy estimate.}
By \eqref{W decomp}, 
the energy estimate for the system $\pa_t (u,v) = W(u,v)$ 
on the real subspace $\{ v = \bar u \}$
becomes 
\begin{equation} \label{energy.est}
\pa_t ( \| u \|_s^2 ) 
= \la \Lm^s \pa_t u , \Lm^s v \ra + \la \Lm^s u, \Lm^s \pa_t v \ra 
= Z_6(u) + Z_{\geq 8}(u)
\end{equation}
where
\begin{align*}
Z_6(u) & 
:= \la \Lm^s (W_5)_1(u,v) , \Lm^s v \ra + \la \Lm^s u, \Lm^s (W_5)_2(u,v) \ra, 
\\ 
Z_{\geq 8}(u)  
& := \la \Lm^s (W_{\geq 7})_1(u,v) , \Lm^s v \ra + \la \Lm^s u, \Lm^s (W_{\geq 7})_2(u,v) \ra, 
\end{align*}
because the term $\big( 1 + \mP(\Phi^{(5)}(u,v)) \big) \big( \mD_1 (u,v) + X^+_3(u,v) \big)$ 
gives zero contribution.  
By Lemma \ref{lemma:stime W geq 7}, one has
\[
| Z_{\geq 8}(u)| \leq C \| u \|_{m_1}^6 \| u \|_s^2.
\]
By \eqref{W5 comp 1 after NF}-\eqref{W5 comp 2 after NF}, we calculate 
\begin{align}
Z_6(u) & = \frac{3i}{32} 
\sum_{\begin{subarray}{c} j,\ell,k \\ |k| = |j| + |\ell| \end{subarray}} 
( u_j u_{-j} u_\ell u_{-\ell} v_k v_{-k} 
- v_j v_{-j} v_\ell v_{-\ell} u_k u_{-k} ) 
|j| |\ell| |k|^{1+2s}
\label{18dic.1} \\ & \quad 
+ \frac{3i}{16}
\sum_{\begin{subarray}{c} j,\ell,k \\ |k| = |j| - |\ell| \end{subarray}} 
( u_j u_{-j} v_\ell v_{-\ell} v_k v_{-k} 
- v_j v_{-j} u_\ell u_{-\ell} u_k u_{-k} ) 
|j| |\ell| |k|^{1+2s},
\label{18dic.2}
\end{align}
which is the sum of the second and the fourth sums in both $(W_5)_1$ and $(W_5)_2$, 
because the first and third sums in $(W_5)_1$ and $(W_5)_2$ cancel out.
Then, we note that the sum over $|k| = |j| - |\ell|$ in \eqref{18dic.2}, 
namely $|j| = |k| + |\ell|$, becomes, after renaming the indices, a sum over 
the same set of indices as the sum in \eqref{18dic.1}. 
Hence 
\begin{align*}
Z_6(u) & = \frac{3i}{32} \sum_{\begin{subarray}{c} j,\ell,k \\ |k| = |j| + |\ell| \end{subarray}} 
( u_j u_{-j} u_\ell u_{-\ell} v_k v_{-k} 
- v_j v_{-j} v_\ell v_{-\ell} u_k u_{-k} ) 
|j| |\ell| |k|   
\big( |k|^{2s} - 2 |j|^{2s} \big),
\end{align*}
namely, symmetrizing $j \leftrightarrow \ell$, 
\begin{align}
Z_6(u) & = \frac{3i}{32} \sum_{\begin{subarray}{c} j,\ell,k \\ |k| = |j| + |\ell| \end{subarray}} 
( u_j u_{-j} u_\ell u_{-\ell} v_k v_{-k} 
- v_j v_{-j} v_\ell v_{-\ell} u_k u_{-k} ) 
|j| |\ell| |k|
\big( |k|^{2s} - |j|^{2s} - |\ell|^{2s} \big).
\label{18dic.4}
\end{align}
For $s = \frac12$, one has 
$|k|^{2s} - |j|^{2s} - |\ell|^{2s} 
= |k| - |j| - |\ell| 
= 0$ 
over the sum, and therefore 
$Z_6(u)$ vanishes for $s=\frac12$. 
Hence $\la \Lm u, v \ra$ 
is a prime integral up to homogeneity order $8$, 
namely  
\[
| \pa_t ( \| u \|_{\frac12}^2 ) |
= | \pa_t \la \Lm u, v \ra |
\leq C \| u \|_{m_1}^6 \| u \|_{\frac12}^2. 
\]
This is not surprising, since $s = \frac12$ in \eqref{energy.est} 
corresponds to the norm in the energy space $H^1 \times L^2$ 
of the original variables, and that norm is controlled by the Hamiltonian.

For $s \neq \frac12$, in general the term $Z_6(u)$ is not zero. 
For example, for $s = 1$ one has 
$|k|^{2s} - |j|^{2s} - |\ell|^{2s} 
= (|j| + |\ell|)^2 - |j|^2 - |\ell|^2
= 2 |j| |\ell|$.

\bigskip

\noindent
\emph{Spheres in Fourier space.}  
We observe that the system 
(or some relevant aspects of it concerning the evolution of Sobolev norms)
can be described by taking sums over all frequencies $k \in \Z^d$ 
with a fixed (Euclidean) length $|k| = \lm$. 
For each $\lm$ in the set
\begin{equation} \label{def set lm}
\Gamma := \{ |k| : k \in \Z^d, \ k \neq 0 \} \subset [1,\infty),
\end{equation}
let
\[
S_\lm := \sum_{k : |k| = \lm} |u_k|^2 
= \sum_{k : |k| = \lm} u_k v_{-k}, 
\quad 
B_\lm := \sum_{k : |k| = \lm} u_k u_{-k}, 
\]
so that 
\[
\overline{B_\lm} = \sum_{k : |k| = \lm} v_k v_{-k},
\quad \ 
\| u \|_s^2 = \sum_{\lm \in \Gamma} \lm^{2s} S_\lm. 
\]
For each $\lm \in \Gamma$, $S_\lm \geq 0$ and $B_\lm \in \C$.
By \eqref{eq for uk}-\eqref{eq for vk},
neglecting the terms from $W_{\geq 7}$, 
one has 
\begin{align}
\pa_t S_\lm 
& 
= \frac{3 i}{32}  \sum_{\begin{subarray}{c} \a,\b \in \Gamma \\  \a + \b = \lm \end{subarray}} 
( B_\a B_\b \overline{B_\lm}  
- \overline{B_\a} \overline{B_\b} B_\lm ) 
\a \b \lm
+ \frac{3 i}{16}  \sum_{\begin{subarray}{c} \a,\b \in \Gamma \\  \a - \b = \lm \end{subarray}}
( B_\a \overline{B_\b} \overline{B_\lm} 
- \overline{B_\a} B_\b B_\lm ) 
\a \b \lm 	
\label{shell.1}
\end{align}
and
\begin{align}
\pa_t B_\lm 
& = - 2 i (1 + \mP) \Big( \lm + \frac14 \lm^2 S_\lm \Big) B_\lm 
+ \frac{i}{16} \sum_{\a \in \Gamma} |B_\a|^2 B_\lm \a^2 
\Big( \frac{1}{\a + \lm} - \frac{1 - \delta_\a^\lm}{\a - \lm} \Big)  
\notag \\ & \quad 
+ \frac{3i}{16} \sum_{\begin{subarray}{c} \a,\b \in \Gamma \\  \a + \b = \lm \end{subarray}} 
B_\a B_\b S_\lm \lm \a \b
+ \frac{i}{8} \sum_{\a \in \Gamma} S_\a S_\lm B_\lm \lm^2 \a 
\Big( 6 + \frac{\a}{\a + \lm} 
+ \frac{\a (1 - \delta_\a^\lm)}{\a - \lm} \Big) 
\notag \\ & \quad 
+ \frac{3i}{8} \sum_{\begin{subarray}{c} \a,\b \in \Gamma \\  \a - \b = \lm \end{subarray}} 
B_\a \overline{B_\b} S_\lm \a \b \lm.
\label{shell.2}
\end{align}
Equations \eqref{shell.1}-\eqref{shell.2} form a closed system in the 
variables $(S_\lm, B_\lm)_{\lm \in \Gamma}$. 
They play the role of an ``effective equation'' 
for the dynamics of the Kirchhoff equation. 
This will be the starting point for further analysis 
in the forthcoming paper \cite{K boot}.

\bigskip

\begin{small}

\textbf{Pietro Baldi}

Dipartimento di Matematica e Applicazioni ``R. Caccioppoli''

Universit\`a di Napoli Federico II

Via Cintia, Monte S. Angelo

80126 Napoli

pietro.baldi@unina.it

\bigskip

\textbf{Emanuele Haus}

Dipartimento di Matematica e Fisica

Universit\`a Roma Tre

Largo San Leonardo Murialdo 1

00146 Roma

ehaus@mat.uniroma3.it

\end{small}


\begin{thebibliography}{99}

\bibitem{Arosio 1993} 
A. Arosio, 
\emph{Averaged evolution equations. The Kirchhoff string and its treatment in scales of Banach spaces}, in: 2nd Workshop on functional-analytic methods in complex analysis (Trieste, 1993), 
World Scientific, Singapore.

\bibitem{Arosio Panizzi 1996} 
A. Arosio, S. Panizzi,
{\em On the well-posedness of the Kirchhoff string},
Trans. Amer. Math. Soc. 348 (1996), 305-330.

\bibitem{Baldi 2009}
P. Baldi, 
\emph{Periodic solutions of forced Kirchhoff equations}, 
Ann. Sc. Norm. Sup. Pisa, Cl. Sci. (5), Vol. VIII (2009), 117-141.

\bibitem{Kold} P. Baldi, E. Haus, 
\emph{On the existence time for the Kirchhoff equation with periodic boundary conditions}, 
Nonlinearity 33 (2020), no. 1, 196-223.

\bibitem{K boot} P. Baldi, E. Haus, 
\emph{Longer lifespan for many solutions of the Kirchhoff equation}, 
in preparation.

\bibitem{Bambusi 2005}
D. Bambusi,
\emph{Galerkin averaging method and Poincar\'e normal form for some quasilinear PDEs},
Ann. Sc. Norm. Super. Pisa Cl. Sci. (5) 4 (2005), 669-702. 

\bibitem{Bernstein 1940} 
S.N. Bernstein, 
\emph{Sur une classe d'\'equations fonctionnelles aux d\'eriv\'ees partielles}, 
Izv. Akad. Nauk SSSR Ser. Mat. 4 (1940), 17-26. 

\bibitem{Berti Delort}
M. Berti, J.-M. Delort,
\emph{Almost global solutions of capillary-gravity water waves equations on the circle}.
Lecture Notes Unione Mat. Ital. 24, Springer, 2018. 
  
\bibitem{Berti Feola Pusateri}  
M. Berti, R. Feola, F. Pusateri,   
\emph{Birkhoff normal form and long time existence for periodic gravity water waves},
arXiv:1810.11549.

\bibitem{Berti Feola Pusateri nota}  
M. Berti, R. Feola, F. Pusateri,   
\emph{Birkhoff normal form for gravity water waves},
Water Waves (2020), online first.

\bibitem{CKSTT 2010} J. Colliander, M. Keel, G. Staffilani, H. Takaoka, T. Tao, 
\emph{Transfer of energy to high frequencies in the cubic defocusing nonlinear Schr\"odinger equation}, Invent. Math. 181 (2010), 39-113.

\bibitem{Corsi Montalto 2018}
L. Corsi, R. Montalto, 
\emph{Quasi-periodic solutions for the forced Kirchhoff equation on $\T^d$}, 
Nonlinearity 31 (2018), no.\,11, 5075-5109.

\bibitem{Craig Sulem 2016 BUMI}
W. Craig, C. Sulem, 
\emph{Mapping properties of normal forms transformations for water waves}, 
Boll. Unione Mat. Ital. 9 (2016), 289-318. 

\bibitem{Craig Worfolk}
Craig W., Worfolk P.,
\newblock {\it An integrable normal form for water waves in infinite depth}.
\newblock { Phys. D}, 84 (1995), no. 3-4, 513-531. 

\bibitem{D'Ancona Spagnolo 1993} 
P. D'Ancona, S. Spagnolo, 
\emph{A class of nonlinear hyperbolic problems with global solutions}, 
Arch. Rational Mech. Anal. 124 (1993), 201-219. 

\bibitem{Delort 2009}
J.-M. Delort, 
\emph{Long-time Sobolev stability for small solutions of quasi-linear Klein-Gordon equations on the circle}, 
Trans. Amer. Math. Soc. 361 (2009), no.\,8, 4299-4365. 
 
\bibitem{Delort 2012}
J.-M. Delort, 
\emph{A quasi-linear Birkhoff normal forms method. 
Application to the quasi-linear Klein-Gordon equation on $\mathbb{S}^1$}, 
Ast\'erisque 341 (2012). 


\bibitem{Dickey 1969} 
R.~W. Dickey,
{\em Infinite systems of nonlinear oscillation equations related to the string},
Proc. Amer. Math. Soc. 23 (1969), 459-468.

\bibitem{Feola Iandoli}
R. Feola, F. Iandoli,
{\em Long time existence for fully nonlinear NLS with small Cauchy data on the circle},
Ann. Sc. Norm. Super. Pisa Cl. Sci. (2020), published online.

\bibitem{Grebert Thomann 2012}
B. Gr\'ebert, L. Thomann,
\emph{Resonant dynamics for the quintic nonlinear Schr\"{o}dinger equation},
\newblock{Ann. Inst. H. Poincar\'e Anal. Non Lin\'eaire} 29 (2012), 455-477.

\bibitem{Greenberg Hu 1980}
J.M. Greenberg, S.C. Hu,
\emph{The initial value problem for a stretched string},
Quart. Appl. Math. 38 (1980/81), 289-311. 

\bibitem{GHHMP 2018}
M. Guardia, Z. Hani, E. Haus, A. Maspero, M. Procesi,
\emph{Strong nonlinear instability and growth of Sobolev norms near quasiperiodic finite-gap tori for the 2D cubic NLS equation},
preprint 2018 (arxiv:1810.03694).

\bibitem{Guardia Haus Procesi 2016}
M. Guardia, E. Haus, M. Procesi, 
\emph{Growth of Sobolev norms for the analytic NLS on $\T^2$}, 
Adv. Math. 301 (2016), 615-692.

\bibitem{Guardia Kaloshin 2015}
M. Guardia, V. Kaloshin, 
\emph{Growth of Sobolev norms in the cubic defocusing nonlinear Schr\"odinger equation}, 
J. Eur. Math. Soc. (JEMS) 17 (2015), 71-149.

\bibitem{Haus Procesi 2015}
E. Haus, M. Procesi, 
\emph{Growth of Sobolev norms for the quintic NLS on $\T^2$}, 
Anal. PDE 8 (2015), 883-922.

\bibitem{Haus Procesi CMP 2017}
E. Haus, M. Procesi, 
\emph{KAM for beating solutions of the quintic NLS}, 
Comm. Math. Phys. 354 (2017), 1101-1132.

\bibitem{Haus Thomann 2013}
E. Haus, L. Thomann,
\emph{Dynamics on resonant clusters for the quintic nonlinear Schr\"{o}dinger equation},
Dyn. Partial Differ. Equ. 10 (2013), 157-169.

\bibitem{Ifrim Tataru 2017}
M. Ifrim, D. Tataru, 
\emph{The lifespan of small data solutions in two dimensional capillary water waves},
Arch. Ration. Mech. Anal. 225 (2017), no. 3, 1279-1346. 

\bibitem{Kirchhoff 1876} 
G. Kirchhoff, 
\emph{Vorlesungen \"uber mathematische Physik: Mechanik}, ch.29, Teubner, Leipzig, 1876.

\bibitem{Lions 1978} 
J.L. Lions, 
\emph{On some questions in boundary value problems of mathematical physics}, 
in: Contemporary developments in continuum mechanics and PDE's, G.M. de la Penha \& L.A. Medeiros eds., North-Holland, Amsterdam, 1978.

\bibitem{Matsuyama Ruzhansky 2015}
T. Matsuyama, M. Ruzhansky, 
\emph{Global well-posedness of the Kirchhoff equation and Kirchhoff systems}, 
Analytic methods in interdisciplinary applications, 81-96, 
Springer Proc. Math. Stat., 116, Springer, Cham, 2015. 

\bibitem{Montalto 2017 KAM K forced}
R. Montalto, 
\emph{Quasi-periodic solutions of forced Kirchhoff equation}, 
NoDEA Nonlinear Differential Equations Appl. 24 (2017), Art. 9.

\bibitem{Spagnolo 1994} 
S. Spagnolo, 
\emph{The Cauchy problem for Kirchhoff equations}, 
Rend. Sem. Mat. Fis. Milano 62 (1994), 17-51.

\end{thebibliography}
\end{document}